\documentclass[12pt]{amsart}
\usepackage[mathscr]{eucal}
\usepackage{amssymb,graphics}
\usepackage{overpic}
\usepackage{caption}
\usepackage{subcaption}

\usepackage{marginnote}

\newtheorem{thm}{Theorem}[section]
\newtheorem*{thm*}{Theorem}
\newtheorem{lemma}[thm]{Lemma}
\newtheorem{prop}[thm]{Proposition}
\newtheorem{cor}[thm]{Corollary}

\theoremstyle{definition}
\newtheorem{defn}[thm]{Definition}

\begin{document}
 
\renewcommand{\H}{\mathcal{H}} 
\newcommand{\cH}{\overline{\H}} 
\newcommand{\h}{\mathfrak{h}}     
\newcommand{\Ht}{\mathsf{H}^2} 
\newcommand{\Htf}{\mathsf{H}^2_{\mathsf{fut}}} 
\newcommand{\R}{\mathbb R}
\newcommand{\E}{\mathsf E}
\newcommand{\Z}{\mathbb Z}
\newcommand{\V}{\mathsf V}
\renewcommand{\P}{\mathsf P}
\newcommand{\CP}{\mathcal{C}} 
\newcommand{\GL}{\mathsf{GL}}
\newcommand{\GLp}{\GL^+}
\newcommand{\PSLtR}{\mathsf{PSL}(2,\R)}

\newcommand{\Aff}{\mathsf{Aff}^+}
\newcommand{\Conf}{\mathsf{Conf}^+}
\newcommand{\Isom}{\mathsf{Isom}^+}
\newcommand{\Confo}{\mathsf{Conf}^0}

\newcommand{\Ise}{\Isom(\E)}

\newcommand{\W}{\mathcal{W}}  
\newcommand{\Wp}{\W_+}
\newcommand{\Wm}{\W_-}

\renewcommand{\L}{\mathsf{L}}  

\renewcommand{\AA}{\mathbb{A}}  
\newcommand{\Fut}{\mathsf{Future}}
\newcommand{\Det}{\mathsf{Det}}
\newcommand{\A}{\mathsf{A}}

\newcommand{\vu}{\mathbf{u}}
\newcommand{\vv}{\mathbf{v}}
\newcommand{\vw}{\mathbf{w}}
\newcommand{\vs}{{\mathbf{s}}}
\newcommand{\vt}{{\mathbf{t}}}
\newcommand{\vn}{{\mathbf{n}}}
\newcommand{\va}{{\mathbf{a}}}
\newcommand{\vb}{{\mathbf{b}}}
\newcommand{\vc}{{\mathbf{c}}}
\newcommand{\vx}{{\mathbf{x}}}
\newcommand{\vy}{{\mathbf{y}}}
\newcommand{\ve}{{\mathbf{e}}}

\newcommand{\Stem}{\mathsf{Stem}}
\newcommand{\xo}[1]{#1^0}
\newcommand{\xp}[1]{#1^+}
\newcommand{\xm}[1]{#1^-}
\newcommand{\xpm}[1]{#1^{\pm}}

\newcommand{\Oto}{\mathsf{O}(2,1)}
\newcommand{\SOto}{\mathsf{SO}(2,1)}
\newcommand{\SOoto}{\mathsf{SO}^{0}(2,1)}
\newcommand{\SOoo}{\mathsf{SO}^{0}(1,1)}

\newcommand{\Sim}{\mathsf{Sim}(\H)}
\newcommand{\IsE}{\mathsf{Isom}^+(\E)}
\newcommand{\ConfE}{\mathsf{Conf}^+(\E)}
\newcommand{\Iso}{\mathsf{Isom}(\H)}
\newcommand{\Isoo}{\mathsf{Isom}^0(\H)}
\newcommand{\Ss}{\mathcal{S}} 
\newcommand{\clSs}{\bar{\Ss}}
\newcommand{\rpo}{\mathbb{RP}^1}
\newcommand{\Rplus}{\R^+}

\newcommand{\Vertex}{\mathsf{Vertex}} 
\newcommand{\Quad}{\mathsf{Quad}} 

\newcommand{\Qu}{\mathsf{Q}} 

\newcommand{\Id}{\mathbb{Id}}
\newcommand{\Fix}{\mathsf{Fix}}
\newcommand{\Bb}{\mathscr{B}}

\newcommand{\ldot}[2]{#1 \cdot #2 }
\newcommand{\lcross}[2]{#1 \times #2 }
\newcommand{\spn}{\mathsf{span}}
\newcommand{\vsc}{\vs^{*}_{2,1}}

\renewcommand\mod[1]{\big(\mathrm{mod}~#1\big)}
\newcommand{\origin}{ \mathsf{0}}
\newcommand{\MH}[1]{h^-_{#1}} 

\newcommand{\Fl}{\mathscr{F}(\ell)}

\title
{Crooked halfspaces}

\author[Burelle]{Jean-Philippe Burelle}
\address{Department of Mathematics\\ University of Maryland\\
    College Park, MD 20742 USA}
    \email{jburelle@umd.edu}
\author[Charette]{Virginie Charette}
    \address{D\'epartement de math\'ematiques\\ Universit\'e de Sherbrooke\\
    Sherbrooke, Quebec, Canada}
    \email{v.charette@usherbrooke.ca}
\author[Drumm]{Todd A. Drumm}
    \address{Department of Mathematics\\ Howard University\\
    Washington, DC }
    \email{tdrumm@howard.edu}
\author[Goldman]{William M. Goldman}
    \address{Department of Mathematics\\ University of Maryland\\
    College Park, MD 20742 USA}
    \email{wmg@math.umd.edu}

\date{\today}
\begin{abstract}
We develop the Lorentzian geometry of a crooked halfspace
in $2+1$-dimensional Minkowski space.
We calculate the affine, conformal and isometric automorphism groups of a crooked halfspace, and discuss its stratification into orbit types, giving an explicit
slice for the action of the automorphism group.
The set of  parallelism classes of timelike lines, or \emph{particles}, in a crooked halfspace is a geodesic halfplane in the hyperbolic plane.
Every point in an open crooked halfspace lies on a particle. 
The correspondence between crooked halfspaces and halfplanes in hyperbolic
$2$-space preserves the partial order defined by inclusion, and the involution
defined by complementarity. 
We find conditions for when a particle lies completely in a crooked half space. 
We revisit the disjointness criterion for crooked planes developed by Drumm and
Goldman in terms of the semigroup of translations preserving a crooked halfspace.
These ideas are then applied to describe foliations of Minkowski space by
crooked planes. 
\end{abstract}
\thanks{
{\em Acknowledgement:\/}
Support by the Institut Mittag-Leffler (Djursholm, Sweden) is gratefully
acknowledged, where much of this paper was written and the last author spoke
on this in the Institut Mittag-Leffler Seminar as part of the Spring 2012 
semester program on ``Geometric and Analytic Aspects of Group Theory."
He gratefully acknowledges partial support by National Science Foundation grant DMS-070781, as well as the partial support by the Institute for Advanced Study in 2008, Institut Henri Poincar\'e, and Centre de Recerca Matematica in 2012.
The first two authors gratefully acknowledge partial support by NSERC and FQRNT
as well as NSF grant DMS-070781. 
The third author gratefully acknowledges partial support by the GEAR network 
funded by the National Science Foundation to visit Institut Henri Poincar\'e  in 2012.
Part of this work began with an undergraduate internship of the first author
with the second author at the Universit\'e de Sherbrooke in 2010, and later
with the last author at the Experimental Geometry Lab at the University of Maryland in 2011.
}

\dedicatory{\begin{center}
{\it Dedicated to the memory of Robert Miner}\end{center}}

\maketitle
\tableofcontents

\section{Introduction}

Crooked planes are special surfaces in $2+1$-dimensional
Minkowski space $\E$. 
They were introduced by the third author~\cite{Drumm}
to construct fundamental polyhedra for
nonsolvable discrete groups $\Gamma$ of isometries 
which act properly on all of  $\E$. 
The existence of such groups $\Gamma$ was discovered
by Margulis~\cite{Margulis1,Margulis2} around 1980 
and was quite unexpected
(see Milnor~\cite{Milnor} for a lucid description of this
problem and~\cite{FG}
for related results). 

In this paper we explore the geometry of crooked planes
and the polyhedra which they bound. 

The basic object is a {\em (crooked) halfspace.\/}
A {\em halfspace\/}  is one of the two components
of the complement of a crooked plane $\CP\subset\E$.
It is the interior of a $3$-dimen\-sion\-al
submanifold-with-boundary, 
and the boundary $\partial\H$ equals $\CP$.

Every crooked halfspace $\H$ determines a {\em halfplane\/}
$\h\subset\Ht$, consisting of directions of 
timelike lines completely contained in $\H$.
Two half\-spaces determine the same 
halfplane if and only if they are {\em parallel,\/}
that is, they differ by a translation. 
The translation is just the unique translation between
the respective vertices of the halfspaces.
We call $\h$ the {\em linearization\/} of $\H$ and denote it $\h = \L(\H)$.
The terminology is motivated by the fact that the linear holonomy
of a complete flat Lorentz manifold defines a complete hyperbolic surface.
See \S\ref{sec:Particles} for a detailed explanation.
In our previous work  \cite{CharetteDrummGoldman2,
CharetteDrummGoldman3,CharetteDrummGoldman4},
we have used crooked planes to extend constructions in $2$-dimensional hyperbolic geometry to Lorentzian $3$-dimensional geometry.

The set $\mathfrak{S}(\Ht)$  of halfplanes in $\Ht$ enjoys a partial ordering given by inclusion
and an involution given by the operation of taking the complement.
Similarly the set $\mathfrak{S}(\E)$  of crooked halfspaces in $\E$ 
is a partially ordered set with involution.
\begin{thm*} 
Linearization $\mathfrak{S}(\E)\xrightarrow{\L} \mathfrak{S}(\Ht)$  
preserves the partial relation defined by inclusion and the involution
defined by complement.
\end{thm*}

Furthermore, we show that any point in a crooked halfspace $\H$ 
lies on a particle determining a timelike direction in the halfplane 
$\L(\H)\subset \Ht$.

Crooked halfspaces enjoy a high degree of symmetry, which we exploit
for the proofs of these results.
In this paper we consider automorphisms preserving the Lorentzian
structure up to isometry, the Lorentzian structure up to conformal 
equivalence, and the underlying affine connection.
\begin{thm*}
Let $\H\subset\E$ be a crooked halfspace.
Its respective groups of orientation-preserving
affine, conformal and isometric automorphisms are:
\begin{align*}
\Aff(\H) & \cong \R^3 \rtimes  \Z/2 \\
\Conf(\H) & \cong \R^2 \rtimes  \Z/2 \\
\Isom(\H) & \cong \R^1 \rtimes  \Z/2 
\end{align*}
The involutions preserving $\H$ are reflections in tachyons orthogonal
to the spine of $\H$, which preserve orientation on $\E$ but reverse 
time-orientation.
\end{thm*}

A fundamental notion in crooked geometry is the {\em stem quadrant\/}
$\Quad(\H)$
of a halfspace $\H$, related the subsemigroup $\V(\H)$ of $\V$ consisting
of translations preserving $\H$. The disjointness results of~\cite{DrummGoldman} can be easily expressed
in terms of this cone of translations.
In particular we prove:
\begin{thm*}
 Two crooked halfspaces are disjoint if and only if
\begin{equation}
\Vertex(\H_1) - \Vertex(\H_2) \in 
\V(\H_1) - \V(\H_2).
\end{equation}
\end{thm*}

Finally these ideas are exploited to construct foliations of $\E$ by
crooked planes. Following our basic theme, we begin with a geodesic foliation of $\Ht$ and extend it to a crooked foliation of $\E$. Such 
foliations may be useful in understanding the deformation theory and the geometry of Margulis spacetimes. 

Figure~\ref{fig:CP} illustrates a crooked plane and the halfspaces which
it bounds.

\newpage
\begin{figure}
\centerline{\includegraphics{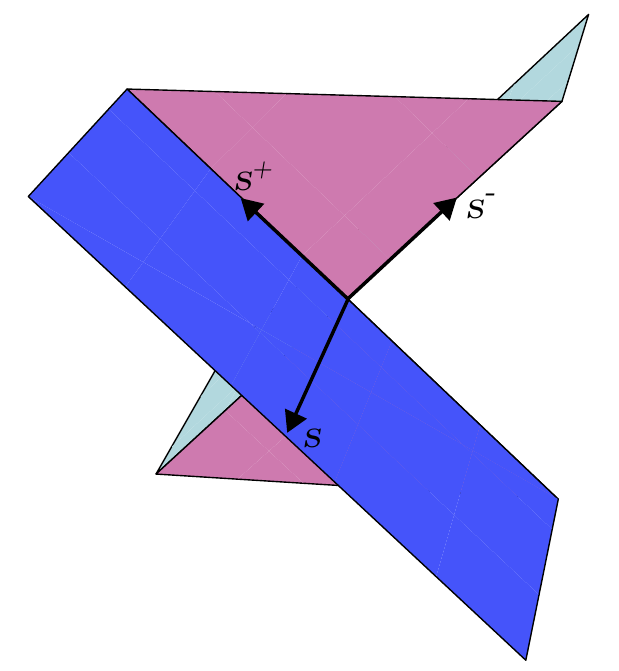}}
\caption{A crooked plane. Wings are halfplanes tangent to the null cone, 
and the stem are the two infinite timelike triangles.
The hinges bound the stem and the wings, and are parallel to the null
vectors $\xm{\vs},\xp{\vs}$. The spacelike vector $\vs$ is parallel to the
spine.}
\label{fig:CP}
\end{figure}

\section{Lorentzian geometry}

\subsection{$(2+1)$-dimensional Minkowski space}
A {\em Lorentzian vector space of dimension 3\/} is a real $3$-dimensional vector space $\V$
endowed with an inner product of signature $(2,1)$. 
The Lorentzian inner product will be denoted:
\begin{align*}
\V \times \V &\longrightarrow \R \\
(\vv, \vu) &\longmapsto  \ldot{\vv}{\vu}.
\end{align*}
We also fix an orientation on $\V$.
The orientation determines a nondegenerate alternating trilinear form
$$
\V\times\V\times\V \xrightarrow{\Det} \R
$$
which takes a positively oriented orthogonal basis $\ve_1,\ve_2,\ve_3$
with inner products
$$
\ldot{\ve_1}{\ve_1} = \ldot{\ve_2}{\ve_2} = 1,\ \ldot{\ve_3}{\ve_3} = -1
$$
to $1$. 
Denote the group of orientation-preserving linear automorphisms of $\V$ 
by $\GLp(3,\R)$.

The oriented Lorentzian $3$-dimensional vector space determines an alternating bilinear mapping
$\V\times\V\longrightarrow\V$, 
called the {\em Lorentzian cross-product,\/}
defined by
\begin{equation}\label{eq:cross}
\Det(\vu,\vv,\vw) = \ldot{\lcross{\vu}{\vv}}{\vw}.
\end{equation}
Compare, for example, with~\cite{DrummGoldman}.


In this paper, {\em Minkowski space\/} $\E$ will mean  a $3$-dimensional 
oriented geodesically complete $1$-connected flat Lorentzian manifold.
It is naturally an affine space having as its group of translations an oriented $3$-dimensional Lorentzian vector space $\V$.
Two points $p,q\in\E$ differ by a unique translation $\vv\in\V$, that is, there is
a unique vector $\vv$ such that
$$
\vv := 
p-q \in\V.
$$ 
We also write $p = q+\vv.$ 
Identify $\E$ with $\V$ by choosing a distinguished point
$o\in\V$, which we call an {\em origin.\/} 
For any point  $p\in\E$ there is a unique vector $\vv\in\V$ such that $p= o +\vv$.
Thus the choice of origin defines a bijection 
\begin{align*}
\V &\xrightarrow{A_o} \E \\
\vv  &\longmapsto  o + \vv.
\end{align*}
For any $o_1,o_2\in\E$, 
$$
A_{o_1} (\vv) = A_{o_2} \big(\vv + (o_1-o_2)\big)
$$
where
$o_1-o_2\in\V$
is the unique vector translating $o_2$ to $o_1$.
A transformation $\E\xrightarrow{T}\E$ normalizes
the group $\V$ of translations if and only if it is 
{\em affine,\/} that is, there is a linear transformation
(denoted $\L(T)$, and called its {\em linear part\/})
such that, for a choice $o$ of origin, 
$$
T(p) = o + \L(T)(p - o) + \vu
$$
for some vector $\vu\in\V$ (called the {\em translational
part of $T$\/}).

\subsection{Causal structure}\label{subsec:causal}
The inner product induces a {\em causal structure\/}  on $\V$: 
a vector $\vv\neq\origin$ is called
\begin{itemize}
 \item \emph{timelike} if $\ldot{\vv}{\vv} < 0$,
 \item \emph{null} (or {\em lightlike\/}) 
 if $\ldot{\vv}{\vv} = 0$,  or

 \item \emph{spacelike} if $\ldot{\vv}{\vv}> 0$.
\end{itemize}
We will call the corresponding subsets of $\V$ respectively 
$\V_-$,$\V_0$ and $\V_+$.
The set $\V_0$ of null vectors is called the {\em light cone}.

Say that vectors $\vu,\vv\in\V$ are 
{\em Lorentzian-perpendicular\/} if $\vu \cdot \vv = 0.$
Denote the linear subspace of vectors Lorentzian-perpendicular
to $\vv$ by $\vv^{\perp}$.  
A line $p+\R\vv$ or ray $p+\Rplus \vv$ is called
\begin{itemize}
 \item a \emph{particle} if $\vv$ is timelike,
 \item a \emph{photon} if $\vv$ is null, and
 \item a \emph{tachyon} if $\vv$ is spacelike.
\end{itemize} 
The set of timelike vectors admits two connected components.
Each component defines a {\em time-orientation} on $\V$.
Since each tangent space $T_p\E$ identifies with $\V$,
the time-orientation on $\V$ naturally carries over to $\E$.
We select one of the components and call it $\Fut$.
Call a non-spacelike vector $\vv \neq \origin$ 
and its corresponding ray {\em future-pointing} 
if $\vv$ lies in the closure of $\Fut$.

The time-orientation can be defined by a choice of a timelike vector
$\vt$ as follows. 
Consider the linear functional $\V \longrightarrow\R$ 
defined by:
$$
\vv\longmapsto \vv \cdot \vt.
$$
Then the future and past components can be distinguished by the sign of
this functional on the set of timelike vectors.

\subsubsection{Null frames}\label{sec:NullFrames}

The restriction of the inner product to the orthogonal complement 
$\vs^{\perp}$ of a spacelike vector $\vs$  is indefinite, having signature
$(1,1)$. The intersection of the light cone with $\vs^{\perp}$ consists of 
two photons intersecting at the origin.
Choose a linearly independent 
pair of future-pointing null vectors $\xpm{\vs}\in \vv^{\perp}$ such that:
$\{ \vs , \xm{\vs} , \xp{\vs} \}$ is a 
positively oriented basis for $\V$
(with respect to a fixed orientation on $\V$). 
The null vectors $\xm{\vs}$ and $\xp{\vs}$ are defined only 
up to positive scaling. 
The standard identity (compare \cite{DrummGoldman} ,
for example), for a unit spacelike vector $\vs$
\begin{align}\label{eq:PlusMinus}
\lcross{\vs}{\xm{\vs}} & =  - \xm{\vs} \notag \\
\lcross{\vs}{\xp{\vs}} & =   \xp{\vs},
\end{align}
will be useful. 

\begin{figure}[ht]
\centerline{\includegraphics{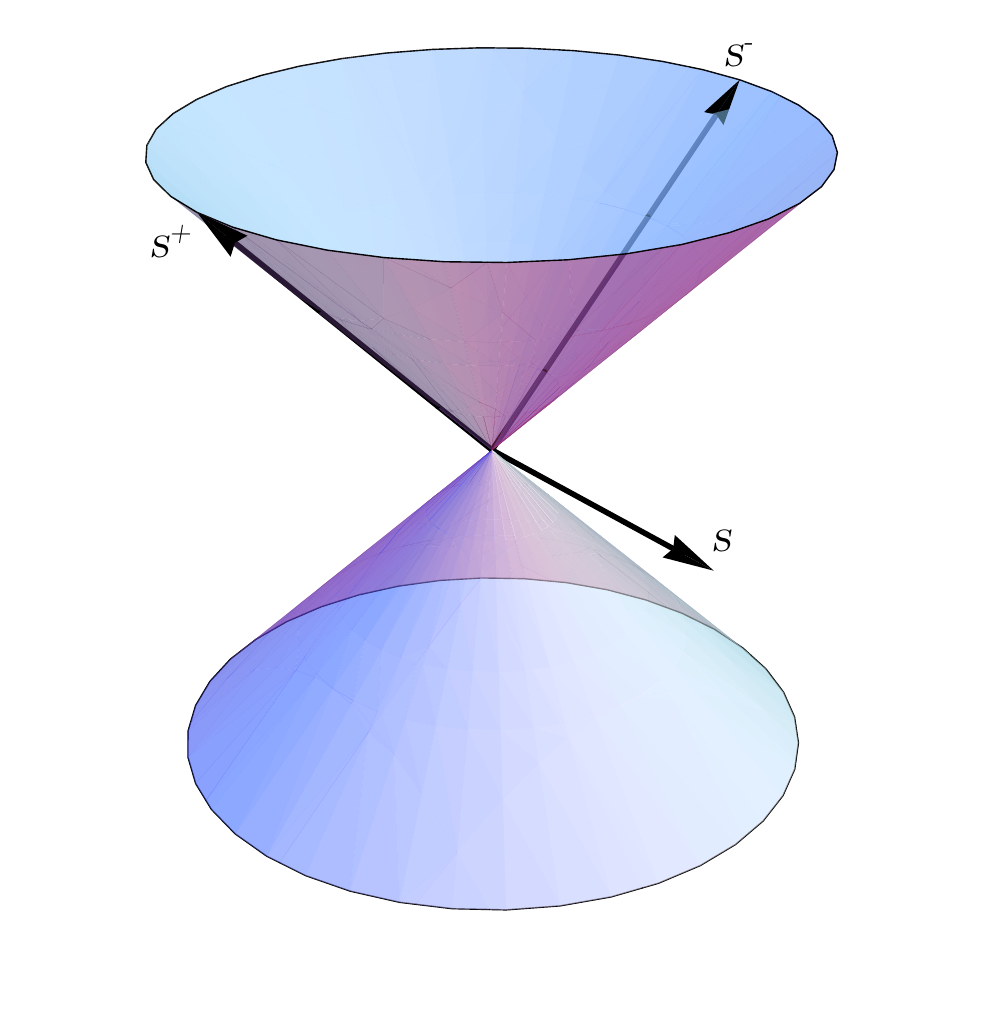}}
\caption{A null frame.}
\label{fig:nullframe}
\end{figure}

We call the positively oriented basis 
$\{ \vs , \xm{\vs} , \xp{\vs} \}$ a
{\em null frame\/} associated to $\vs$.
(Margulis~\cite{Margulis1,Margulis2}  takes the null vectors 
$ \xm{\vs} , \xp{\vs} $ to have unit Euclidean
 length.) %
We  instead require 
that they are future-pointing, 
normalize $\vs$ to be 
{\em unit-spacelike,\/} that is, 
$\ldot{\vs}{\vs} = 1$,
and choose 
$\xm{\vs}$ and $\xp{\vs}$ so that 
$\xm{\vs} \cdot \xp{\vs} = -1.$
In this normalized basis the corresponding
Gram matrix (the symmetric matrix of inner products) has the form
 $$
 \bmatrix 1 & 0 & 0 \\ 0 & 0 & -1 \\ 0 & -1 & 0 \endbmatrix.
 $$
The normalized null frame defines linear coordinates $(a,b,c)$ on $\V$:
$$
\vv := a \vs + b \xm{\vs} + c \xp{\vs} 
$$ 
so that in these coordinates the corresponding Lorentz metric on $\E$ is:
\begin{equation}\label{eq:MetricTensor}
da^2 -db\  dc.
\end{equation}

%

\subsection{Transformations of $\E$}\label{sec:transformations}
The orientation on the vector space
$\V$ defines an orientation
on the manifold $\E$. 
A linear automorphism of $\V$ preserves
orientation if and only if it has positive
determinant. An affine automorphism
of $\E$ preserves orientation if and only
if its linear part lies in the subgroup $\GLp(3,\R)$
of $\GL(3,\R)$ consisting of matrices of positive
determinant.
The group of orientation-preserving
affine automorphisms of $\E$ then
decomposes as a semidirect product:
$$
\Aff(\E)\ =\  \V \rtimes \GLp(3,\R).
$$

Denote the group of orthogonal automorphisms
(linear isometries) of $\V$ by $\Oto$ and the subgroup of
{\em orientation-preserving\/} isometries by $\SOto$.
Let 
$$
\SOto := \Oto \cap \GLp(3,\R)
$$ 
denote, as usual, the subgroup of orientation-preserving
linear isometries.
The group of orientation-preserving linear {\em conformal automorphisms\/}
of $\V$ is the product $\SOto \times \Rplus$, 
where $\Rplus$ is the one-parameter group of
{\em positive homotheties\/} $\vv \mapsto e^s \vv$.
(Compare \eqref{eq:homotheties}.)
Orientation-preserving isometries of $\E$ constitute the subgroup:
$$
\IsE\, : = \,  \V \rtimes \SOto 
$$
and the subgroup of orientation-preserving 
conformal automorphisms is:
$$
\Conf(\E)\ =\  \V \rtimes
\big(\SOto \times \Rplus\big)
$$
\subsubsection{Components of the isometry group}
The group $\Oto$ has four connected components.
The identity component $\SOoto$ consists of orientation-preserving
linear isometries preserving time-orientation.
It is isomorphic to the group $\PSLtR$ of orientation-preserving
isometries of the hyperbolic plane $\Ht$.
(The relationship with hyperbolic geometry will be explored in \S\ref{sec:hyp}.)
The group $\Oto$ is a semidirect product
$$
\Oto \cong  (\Z/2 \times\Z/2) \rtimes \SOoto
$$
where $\pi_0\big(\Oto\big) \cong  \Z/2 \times\Z/2$
is generated by reflection in a point (the antipodal map $\AA$, which
reverses orientation) and reflection in a tachyon (which preserves
orientation, but reverses time-orientation).

\subsubsection{Transvections, boosts, homotheties and reflections}
\label{sec:Tbhr}
In the null frame coordinates of \S\ref{sec:NullFrames} , 
the one-parameter group of linear isometries
\begin{equation}\label{eq:Transvections}
\xi_t \;:=\; \bmatrix 1 & 0 & 0 \\ 0 & e^t & 0 \\ 0 & 0 & e^{-t} \endbmatrix,
\end{equation}
(for $t\in\R$)
fixes $\vs$ and acts on the (indefinite) plane $\vs^\perp$.
These transformations, called {\em boosts,\/} 
constitute the identity component $\SOoo$ of the isometry
group of $\vv^\perp$.
The one-parameter group $\Rplus$ of
{\em positive homotheties\/}
\begin{equation}\label{eq:homotheties}
\eta_s \;:=\; \bmatrix e^s & 0 & 0 \\ 0 & e^s & 0 \\ 0 & 0 & e^s \endbmatrix
\end{equation}
(where $s\in\R$) acts conformally on Minkowski space,
preserving orientation.
The involution 
\begin{equation}\label{eq:rho}
\rho \;:=\; \bmatrix -1 & 0 & 0 \\ 0 & 0& -1 \\ 0 &  -1&0 \endbmatrix.
\end{equation}
preserves orientation, reverses time-orientation, 
reverses $\vs$,
and interchanges the two null lines 
$\R\xm{\vs}$ and $\R\xp{\vs}$.

\subsection{Octants, quadrants, solid quadrants}
The following terminology will be used in the sequel.
A {\em quadrant\/} in a vector space $\V$ is 
the set of nonnegative linear combinations of two linearly independent
vectors. 
A {\em quadrant\/} in an affine space $\E$ is the translate
of a point in $\E$ by a quadrant in the vector space underlying $\E$.
Similarly an {\em octant\/} in a vector space or affine space is obtained 
from nonnegative linear combinations of {\em three\/} linearly independent
vectors. A {\em solid quadrant\/} is the set of linear combinations
$ a \va + b \vb + c \vc$ where $a, b, c \in\R$,  $a,b \ge 0$ and $\va,\vb,\vc$
are linearly independent.

\subsection{Hyperbolic geometry}\label{sec:hyp}
The Klein-Beltrami projective model of hyperbolic geometry identifies
the hyperbolic plane $\Ht$ with the subset 
$\P(\V_-)$ of the real projective
plane $\P(\V)$ corresponding to particles (timelike lines).
Fixing an origin $o\in\E$ identifies the affine (Minkowski) space 
$\E$ with the Lorentzian vector space $\V$.
Thus the hyperbolic plane identifies with particles passing through $o$,
or equivalently translational equivalence classes 
(parallelism classes) of particles in $\E$.

\subsubsection{Orientations in $\Ht$}
An orientation of $\Ht$ is given by  
a {\em time-orientation\/} in 
$\V$, that is, a connected component of $\V_-$, as follows.
The subset
$$
\Htf := \big\{ \vv\in\Fut \mid \vv\cdot\vv = -1 \big\}
$$
of the selected connected component $\Fut$ of $\V_-$
is a cross-section for the  $\Rplus$-action $\V_-$
by homotheties:
the restriction of the quotient mapping
$$ 
\V\setminus\{0\} \longrightarrow \P(\V) 
$$
to $\Htf$ identifies $\Htf \xrightarrow{\cong}\Ht$.

The radial vector field on $\V$ is transverse to the 
hypersurface $\Htf$. %
Therefore the radial %
vector field, together
with the ambient orientation on $\V$, defines an orientation
on $\Ht$. 
Using the {\em past-pointing timelike 
vectors\/} for a model for $\Ht$ along with the fixed orientation of $\V$, 
we would obtain the opposite orientation. 
This follows since the antipodal map on $\V$ relates
future and past, and the antipodal map reverses orientation
(in dimension $3$). 
Fixing a time-orientation and reversing orientation in $\V$ reverses the induced
orientation on $\Ht$.

%
%

\subsubsection{Halfplanes in $\Ht$}\label{sec:hyperbolichalfplanes}
Just as points in $\Ht$ correspond to translational equivalence classes of particles in $\E$, 
geodesics in $\Ht$ 
correspond to %
translational equivalence classes of tachyons in $\E$.  
Given a spacelike vector $\vv\in\V$, the projectivization 
$\P(\vv^\perp)$ meets $\P(\V_-) \approx \Ht$ in a geodesic.
A geodesic in $\Ht$ separates $\Ht$ into two {\em halfplanes.}

With a time-orientation, spacelike vectors in $\V$ conveniently parametrize
halfplanes in $\Ht$.
Using the identification of $\Ht$ with $\Htf$ above, a spacelike vector 
$\vs$ determines a {\em halfplane} in $\Ht$:
$$
\h(\vs) := \big\{ \vv\in\Htf \mid\vv\cdot\vs \ge 0 \big\}
$$
bounded by the geodesic $\P(\vs^\perp)$. 
The complement of the geodesic 
$\P(\vs^\perp)$
in $\Ht$ consists of the interiors of the two halfplanes 
$\h(\vs)$ and $\h(-\vs)$.
%
Furthermore, using the fixed orientation on $\Ht$,  
an {\em oriented geodesic\/} $l\subset\Ht$ determines a halfplane whose boundary is $l$, as follows. Let $u\in l$ be a point and $v$ be the unit vector
tangent to $l$ at $u$ pointing in the {\em forward\/} direction, as determined
by the orientation of $l$. Choose the halfplane bounded by $l$ so that
the pair $(v,n)$ is positively oriented, where $n$ is an inward pointing normal vector to $l$ at $u$.

Transitivity of the action of $\mathsf{Isom}^+(\Ht)$ on 
oriented geodesics implies:

\begin{lemma}\label{lem:HtTransitivity}
The group of orientation-preserving isometries of $\Ht$ acts transitively
on the set of halfplanes in $\Ht$. The isotropy group of a halfplane
$\h(\vs)$ is the one-parameter group of transvections along the geodesic
$\partial\h(\vs)$.
\end{lemma}
\noindent 
In the example \eqref{eq:Transvections} in \S\ref{sec:Tbhr},
where 
\begin{equation}\label{eq:100vector}
\vs \ = \ \bmatrix 1 \\ 0 \\ 0 \endbmatrix,
\end{equation}
this one-parameter group of transvections is just  
$\{\xi_t,t\in\mathbb{R}\}$.

\subsection{Disjointness of halfplanes}

We use a  disjointness criterion for two halfplanes in $\Ht$ in terms
of the following definition:
\begin{defn}
Two spacelike vectors $\vs_1,\vs_2\in\V$ are \emph{consistently oriented} if 
 $\vs_1\cdot\vs_2 < 0$,  $\vs_1\cdot \vs_2^{\pm}\le 0$, and 
$\vs_1^\pm\cdot \vs_2\le 0$.
\end{defn}
\noindent
Given the orientation defined above on $\Ht$, two consistently oriented unit-spacelike
vectors have a useful characterization in terms of halfplanes.
\begin{lemma}\label{lem:consistentorientation}
Let  $\vs_1,\vs_2\in\V$ be spacelike vectors. The vectors $\vs_1$ and $\vs_2$ are consistently oriented if and 
only if the corresponding halfplanes 
$\h(\vs_1)$ and $\h(\vs_2)$ are disjoint.
\end{lemma}
Before describing the proof, we give two simple examples illustrating the 
concept of 
consistent orientation. Consider the unit spacelike vectors
$$
\vs_1 := \bmatrix -1 \\ 0 \\ 0 \endbmatrix, \ 
\vs_2 := \bmatrix \cosh(t) \\ 0 \\ \sinh(t) \endbmatrix
$$
where the Lorentzian structure is defined by the quadratic form
$$
\bmatrix x \\ y \\ z \endbmatrix \longmapsto  x^2 + y^2 - z^2.
$$
Then $\vs_1\cdot \vs_2 = -\cosh(t) \le -1 < 0$ for all $t$.
The condition that $\vs_1\cdot\vs_2 \le -1$ 
ensures that the corresponding geodesics in $\Ht$ are disjoint
(or identical). However, even if the geodesics are ultraparallel or asymptotic,
the halfplanes may be nested or intersect in a slab. The conditions
on $\xpm{\vs_i}$ exclude these cases.

$\Ht$ identifies with the unit disc $x^2 + y^2 < 1$
in the affine hyperplane defined by $z=1$. 
The corresponding halfplanes $\h(\vs_i)$ are then defined by:
\begin{align*}
\h(\vs_i) & = \{ (x,y) \mid  x < 0\} \\
\h(\vs_2) &= \{ (x,y) \mid  x > \tanh(t) \}.
\end{align*}
which are disjoint if and only if $t > 0$. 

Now 
$$
\xpm{\vs_1} = \bmatrix 0 \\ \mp 1 \\ 0 \endbmatrix,\quad
\xpm{\vs_2} = \bmatrix \sinh(t) \\ \pm 1 \\ \cosh(t) \endbmatrix,
$$
so 
$$
\vs_1\cdot \xpm{\vs_2} = 
\vs_2\cdot \xpm{\vs_1} = -\sinh(t).
$$ 
Thus $\h(\vs_1)\cap \h(\vs_2) =\emptyset$ if and only if $\vs_1$ and $\vs_2$
are consistently oriented.

A similar example where the corresponding geodesics are
asymptotic occurs with the same $\vs_1$ but
\begin{equation*}
\vs_2 = \bmatrix 1 \\ 1 \\ 1\endbmatrix.
\end{equation*}
Then $\xm{\vs_2} = \xp{\vs_1}$ above and
\begin{equation*}
\xp{\vs_2} = \bmatrix 1 \\ 0 \\ 1 \endbmatrix
\end{equation*}
and pair $(\vs_1,\vs_2)$ is consistently oriented. The halfplane
$\h(\vs_2)$ is defined by $x + y > 1$ which is disjoint from
$\h(\vs_1)$. Compare Figure~\ref{fig:halfplanes}.

\begin{figure}
  \centering
  \begin{subfigure}[b]{0.4\textwidth}
     \begin{overpic}[scale=.5,tics=10]{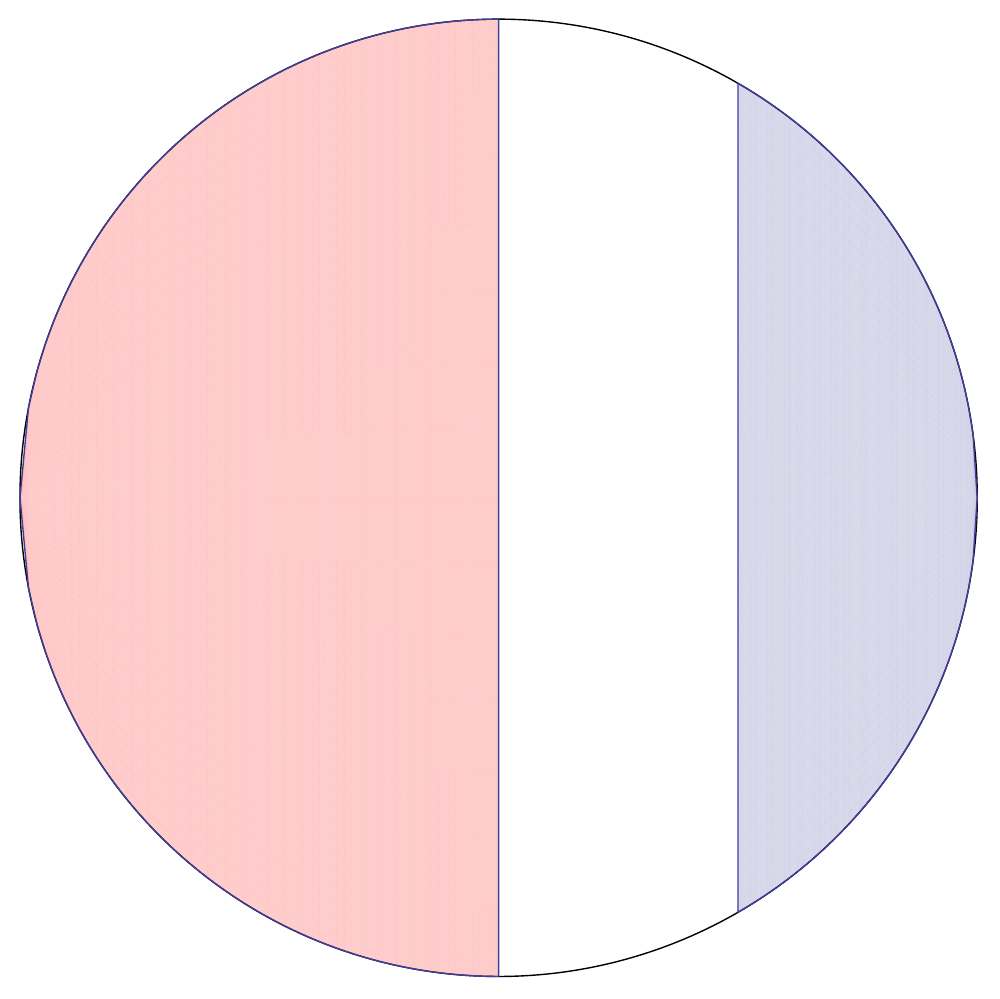}
      \put(20,50){$\mathfrak{h}(s_1)$}
      \put(80,50){$\mathfrak{h}(s_2)$}
     \end{overpic}
  \end{subfigure}
  \begin{subfigure}[b]{0.4\textwidth}
      \begin{overpic}[scale=.5,tics=10]{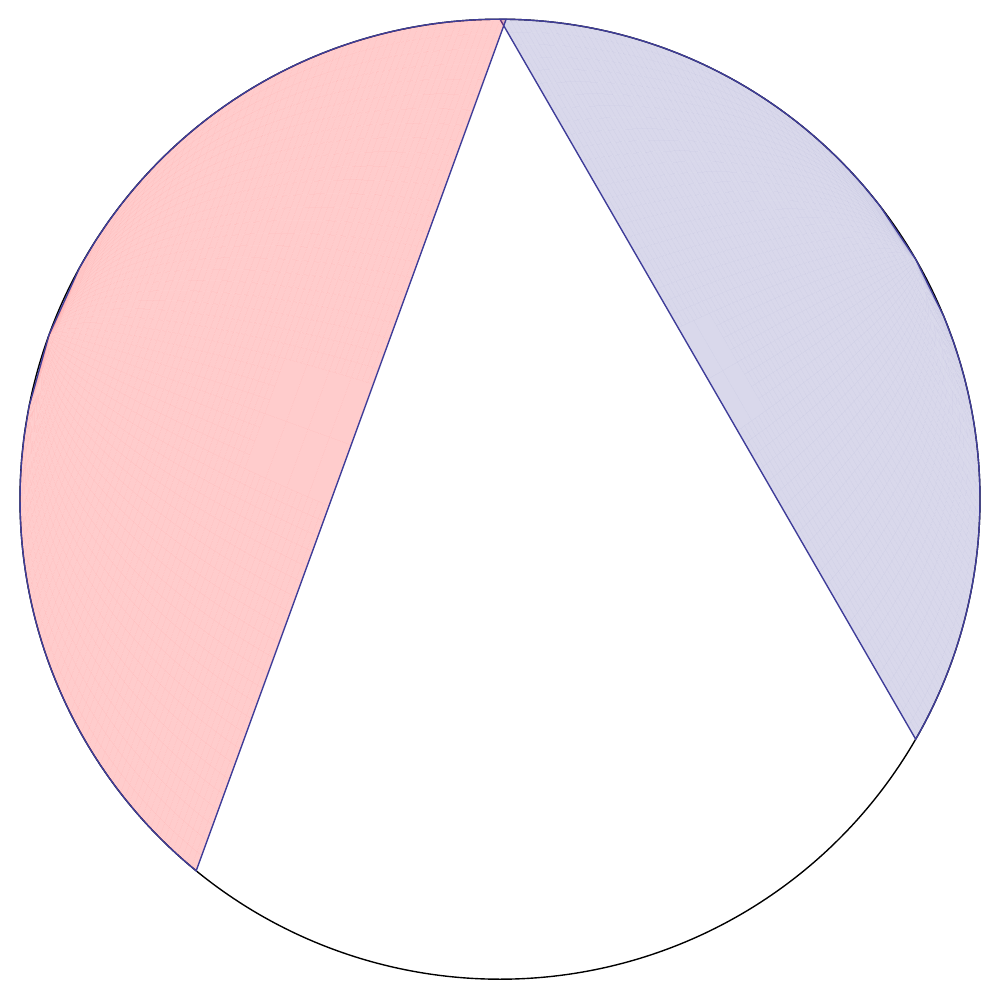}
      \put(20,70){$\mathfrak{h}(s_1)$}
      \put(70,70){$\mathfrak{h}(s_2)$}
      \end{overpic}
  \end{subfigure}
    \caption{Consistently oriented halfplanes. 
  The first picture depicts halfplanes bounded by ultraparallel geodesics,
  and the second picture depicts halfplanes bounded by asymptotic
  geodesics.}
\label{fig:halfplanes}
\end{figure}

%


\begin{proof}[Proof of Lemma~\ref{lem:consistentorientation}]
Given a spacelike vector $\vs_i$, the solid quadrant defined by 
combinations $a \vs_i + b\vs_i^- + c\vs_i^+ $ where $a,b,c\in\R$ and 
$b,c >0$ contains all of the future-pointing timelike vectors. 
Moreover, the octant                                                                                                                         
\[
A(\vs) := \{ a \vs_i + b\vs_i^- + c\vs_i^+  \, \vert \, a,b,c >0 \}
\]
defines the halfplane $\h(\vs_i)$. In particular, 
$$
\h(\vs_i) = A(\vs_i )\cap \Ht .
$$

Suppose that 
$\vs_1$ and $ \vs_2$ are consistently oriented spacelike vectors. By definition, any vector in $A(\vs_1)$ is a
positive linear combination of vectors whose inner product with 
$\vs_2$ is negative, so its inner product with $\vs_2$ is negative. 
Thus
$A(\vs_1)\ \cap\ A(\vs_2)\ =\emptyset$ and 
$\h(\vs_1)\ \cap\ \h(\vs_2)\ =\ \emptyset$.

Now suppose that 
$
\h(\vs_1)\  \cap\ \h(\vs_2)\ =\ \emptyset.
$ 
Then $A(\vs_1)\cap A(\vs_2) = \emptyset$, and  $\vs_i$ 
and $\vs_i^{\pm}$ all have a negative inner product with $\vs_j$,
as desired.
\end{proof}

\section{Crooked halfspaces}
In this section we define crooked halfspaces and describe their basic structure.
Following our earlier papers, we consider an {\em open\/} crooked halfspace
$\H$, denoting its closure by $\cH$ and its complement by $\H^c$.

A crooked halfspace $\H$ is bounded by a crooked plane $\partial\H$.
A crooked plane is a $2$-dimensional polyhedron with $4$ faces, which is
homeomorphic to $\R^2$. It is non-differentiable along two lines (called
{\em hinges\/}) meeting in a point (called the {\em vertex\/}). 
The hinges are null lines bounding null halfplanes in $\E$ 
(called {\em wings\/}).
(Null halfplanes in $\E$ are defined below in \S\ref{sec:NullHalfplanes}.)
The wings are connected by the union of two quadrants in the plane
containing the hinges. Call the plane spanned by the two hinges the
{\em stem plane\/} and denote it $\Ss$.  
The hinges are the only null lines contained in $\Ss.$
The union of the hinges and all timelike lines in $\Ss$ forms the {\em stem.\/} 
The stem plane may be equivalently defined as the unique plane containing the stem.

\subsection{The crooked halfspace\/}
We explicitly compute a crooked halfspace in the 
coordinates $(a,b,c)$ defined in \S\ref{sec:NullFrames}.  Recall that in those coordinates the Lorentzian metric tensor equals $da^2 - db\ dc$.  
Lemma~\ref{lem:transitivity} (discussed in \S\ref{sec:transitivity})
asserts that 
all crooked halfspaces are $\IsE$-equivalent.

\subsubsection{The director and the vertex}
Let $\vs\in\V$ be a (unit-) spacelike vector and $p\in\E$.
Then the {\em (open) crooked halfspace directed by $\vs$ and 
vertexed at $p$\/} is the union:
\begin{align}\label{eq:HSDef}
\H(p,\vs) :=\, 
&
\{ q\in\E \, \vline \,
	\ldot{(q-p)}{\xp{\vs}}   <  0
				  \mbox{ and } 
		    \ldot{(q-p)}{\vs}  > 0  
		    \}  \;  \cup \notag\\	
&\{ 
q\in\E \, \vline \,
	\ldot{(q-p)}{\xp{\vs}}   <  0,
	\ldot{(q-p)}{\xm{\vs}}  > 0,			 
				 \mbox{ and } 
		    \ldot{(q-p)}{\vs}  =  0  %
		    \}  \  \cup \notag\\	
&\{ 
q\in\E \, \vline \,
 		    \ldot{(q-p)}{\xm{\vs}}  > 0  
			\mbox{ and } 
		    \ldot{(q-p)}{\vs}  < 0 
         \} 
\end{align}

\noindent
Its closure is the
{\em closed crooked halfspace\/}
with director $\vs$ and vertex $p$, %
defined as the union:


\begin{alignat}{2}\label{eq:ClHSDef}
\cH(p,\vs) :=\, 
&
\{ q\in\E \, \vline \,
	\ldot{(q-p)}{\xp{\vs}}   \leq 0
				 && \mbox{ and } 
		    \ldot{(q-p)}{\vs}  \geq 0  
		    \}     \notag\\	
 \cup\; &  \{ 
q\in\E \, \vline \,
 		    \ldot{(q-p)}{\xm{\vs}}  \geq 0  &&
			\mbox{ and } 
		    \ldot{(q-p)}{\vs}  \leq 0 
         \} 
\end{alignat}
Write:
$$
p =: \Vertex\big( 
\cH(p,\vs) 
\big) \ = \ 
\Vertex\big( \H(p,\vs)  \big).
$$
The director and vertex of the halfspace are
the unit-space\-like vector and the point:
$$
\vs = \bmatrix 1 \\ 0 \\ 0 \endbmatrix, \quad
p = \bmatrix 0 \\ 0 \\ 0 \endbmatrix,
$$
respectively. 
Null vectors corresponding to $\vs$ are:
$$
\xm{\vs} = \bmatrix 0 \\ 1 \\ 0 \endbmatrix, \quad
\xp{\vs} = \bmatrix 0 \\ 0 \\ 1 \endbmatrix.
$$
so that in above coordinates $\H$ is defined by the inequalities:
\begin{align*}
b > 0 &\ \text{~~if}\  a >  0 \\
b > 0 > c &\ \text{~~if}\  a =  0 \\
0 > c &\ \text{~~if}\  a < 0. 
\end{align*}
The corresponding closed crooked halfspace $\cH$ is defined by:
\begin{align*}
b \ge 0 &\ \text{~~if}\  a >  0 \\
b \ge 0 \text{~or~} 0 \ge c &\ \text{~~if}\  a =  0 \\
0 \ge c &\ \text{~~if}\  a < 0. 
\end{align*}
\noindent
\subsubsection{Octant notation}
The shape of a crooked halfspace suggests the following notation:
$$
\cH = \{ a,b \ge 0 \} \cup \{ a,c \le 0 \}.
$$
The three coordinate planes for $(a,b,c)$  divides $\E$ (identified with $\V$)
into eight open octants, depending on the signs of these three coordinates.
Denote a subset of $\E$ by an ordered triple of symbols such as $+, -, 0, \pm$
to describe whether the corresponding coordinate is respectively positive, negative, zero, or 
arbitrary. For example the positive octant is $(+,+,+)$
and the negative octant is $(-,-,-)$. In this notation, the open crooked halfspace $\H(\vs,p)$ is the union
$$
(+, +,\pm) \cup (0,+,-) \cup (-,\pm,-).
$$

\subsubsection{The hinges and the stem plane}
The {\em hinges\/} of $\H$ are the lines through the vertex parallel to the null vectors $\xm{\vs}, \xp{\vs}$:
\begin{align*}
h_-(p,\vs) &:= p + \R\xm{\vs}, \\ h_+(p,\vs)  &:= p + \R\xp{\vs}.
\end{align*}
so the {\em stem plane\/} (the affine plane spanned by the hinges) equals:
$$
\Ss(p,\vs) := p + \vs^\perp.
$$
In octant notation, $h_-  = (0,\pm,0)$, $h_+ = (0,0,\pm)$ 
and the stem plane is the coordinate plane $(0,\pm,\pm)$ defined by $a = 0$.

\subsubsection{The stem}
The stem consists of timelike directions inside the light cone in the stem plane.
That is,
$$
\Stem(p,\vs) := \{ p + \vv \mid  \vv\in\vs^\perp, \  \vv\cdot \vv \le 0 \}.
$$
In octant notation the stem is $(0,+,+)\cup(0,-,-)$ 
and is defined by:
$$ 
a = 0,\  bc > 0.
$$ 
The stem decomposes into two components: 
a {\em future\/} stem $(0,+,+)$ and a {\em past\/} stem $(0,-,-)$.
Of course the boundary $\partial\Stem$ is the union of the hinges 
$h_-\cup h_+$.

\subsubsection{Particles in the stem}
Particles in the stem also determine involutions which interchange the pair of halfspaces
complementary to $\CP$. Particles are lines
spanned by the future-pointing timelike vectors
$$
\vt_t :=  \bmatrix 0 \\ e^t \\ e^{-t} \endbmatrix ,
$$
for any $t\in\R$, with the corresponding particles defined by $a = 0$, $ b = e^{2t} c$. 
The corresponding reflection is:
$$
\bmatrix a \\ b \\ c \endbmatrix
\stackrel{R_t}\longmapsto 
\bmatrix -a \\  e^{2t} c \\ e^{-2t} b \endbmatrix.
$$
\noindent
See~\cite{Charette3} for a detailed study of 
involutions of $\E$. 

\subsection{The wings}
The wings are defined by a construction (denoted $\W$) involving the orientation of $\V$.
We associate to every null vector $\vn$ a {\em null halfplane\/} 
$\W(\vn)\subset\V$
and to every null line $p + \R\vn$ the affine null halfplane $p + \W(\vn)$.
Define the {\em wings\/} of the halfspace  $\H(\vs,p)$ as $p + \W(\xm{\vs})$
and $p + \W(\xp{\vs})$ respectively.

\subsubsection{Null halfplanes}\label{sec:NullHalfplanes}
Let $\vn$ be a future-pointing null vector. 
Its orthogonal plane $\vn^{\perp}$ is tangent to the light cone.
Then the line $\R\vn$ lies in the plane $\vn^\perp$.
The complement $\vn^\perp \setminus \R\vn$
has two components, called {\em null halfplanes.\/}
Consider a spacelike vector $\vv \in \vn^\perp$. Then $\vn$ is either a multiple of 
$\vv^+$ or $\vv^-$. Two spacelike vectors $\vv,~\vw\in\vn^\perp$ are in the same halfplane if and only if 
$$
\xp{\vv}=\xp{\vw}=\vn \ \text{~or~}\  \xm{\vv}=\xm{\vw}=\vn,
$$
up to scaling by a positive real.  A spacelike vector $\vs\in\V$ thus unambiguously defines 
the following (positively extended) wing:
%
\begin{equation}\label{eq:WingDef}
 \W(\vs) :=\, \{ \vw \in \V \mid  \ldot{\vw}{\vs} \geq 0
	\mbox{ and } \ldot{\vw}{ \xp{\vs}}   =0 \} .
\end{equation}
Each hinge bounds a wing. 
The wings bounded by  the hinges 
$h_- = (0,0,\pm) 
$ and $h_+ = 
(0,\pm,0)$ are defined, respectively, by:
\begin{align*}
\W_- &:= (+,0,\pm) = \{ a \ge 0, b = 0\}, \\
\W_+ &:= (-,\pm,0) = \{ a \le 0, c = 0\} .
\end{align*}

\subsubsection{The spine}\label{sec:spine}
A crooked plane $\CP$ contains a unique tachyon $\sigma$ called its 
{\em spine.\/} The spine is the line through $\Vertex(\CP)$ parallel
to the director of $\CP$.
It lies in the union of the two wings, and is orthogonal to each hinge.
The spine is defined by $b = c = 0$ or $(\pm,0,0)$ in quadrant notation.

Reflection $R$ in the spine interchanges the halfspaces complementary to $\CP$. In the usual coordinates it is:
\begin{equation}\label{eq:spine}
\bmatrix a \\ b \\ c \endbmatrix
\stackrel{R}\longmapsto 
\bmatrix a \\ - b \\ - c \endbmatrix.
\end{equation}
Furthermore each halfspace complementary to $\CP$ is a fundamental domain for
$\langle R\rangle.$ 

\subsubsection{The role of orientation}
The orientation of $\E$ is crucially used to define wings.
Since the group of all automorphisms of $\E$ is a double extension
of the group of orientation-preserving automorphisms by the antipodal
map $\AA$, one obtains a parallel but opposite theory 
by composing with $\AA$. (Alternatively, one could work with negatively
oriented bases to define null frames etc.)
{ \em Negatively extended\/} crooked halfspaces and crooked planes  are defined as in \eqref{eq:HSDef} 
except that all the inequalities involving $\ldot{(q-p)}{\vv}$ are reversed.
In this paper we fix the orientation of $\E$ and thus only consider
positively extended halfspaces.
For more details, see \cite{DrummGoldman}.

\subsection{The bounding crooked plane}
If $\H(p,\vs)$ is a halfspace, 
then its boundary $\partial\H(p,\vs)$ is a crooked plane,  denoted 
$\CP(p,\vs)$.
A crooked plane is the union of its stem and two wings along the hinges
which meet at the vertex. 
Observe that the complement of $\H(p,\vv)$ is the closed crooked halfspace
$\cH(p,-\vv)$ and
$$
\partial\H(p,\vv) = \CP(p,\vv)  = \partial\H(p,-\vv).
$$

\subsection{Transitivity}\label{sec:transitivity}
For calculations it suffices to consider only one example of a crooked halfspace 
thanks to:

\begin{lemma}\label{lem:transitivity}
The group $\IsE$ acts transitively on the set of (positively oriented) crooked halfspaces in $\E$.
\end{lemma}
\begin{proof}
The group $\V$ acts transitively on the set of points $p\in\E$ and by Lemma~\ref{lem:HtTransitivity} 
$\SOto$ acts transitively on the set of unit spacelike vectors $\vs$. Thus $\IsE$ acts transitively 
on the set of pairs $(p,\vs)$ where
$p\in\E$ is a point and $\vs$ is a unit-spacelike vector.
Since such pairs determine crooked half spaces, $\IsE$ acts transitively
on crooked halfspaces.
\end{proof}
In a similar way, the full group of (possibly orientation-reversing) isometries
of $\E$ acts transitively on the set of (possibly negatively extended)
crooked halfspaces.

\subsection{The stem quadrant}
A particularly important part of the structure of a crooked halfspace  $\H$ 
is its {\em stem quadrant $\Quad(\H)$,\/}  
defined as the closure of the intersection of 
$\H$ with its stem plane $\Ss(\H)$ and denoted:
\begin{equation}\label{eq:StemQuadrant}
\Quad(\H) := \overline{\big(\H \cap \Ss(\H)\big)}\ \subset \E.
\end{equation}
Closely related is the {\em translational semigroup $\V(\H)$,\/}
defined as the set of translations preserving $\H$:
\begin{equation}\label{eq:Semigroup}
\V(\H) := \{ \vv\in\V \mid  \H + \vv \subset \H \} \ \subset \V.
\end{equation}
\bigskip\bigskip\bigskip

\begin{figure}[hb]
\centerline{\includegraphics{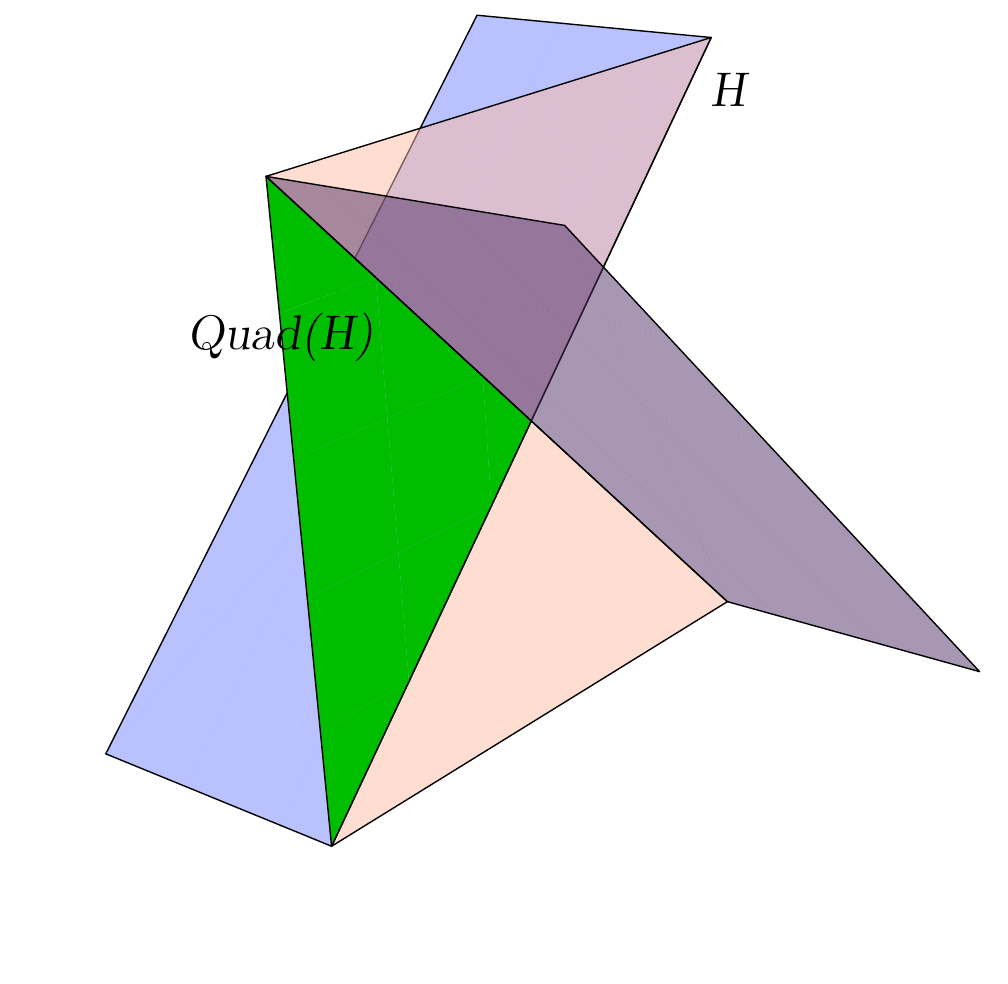}}
\caption{The stem quadrant of a crooked halfspace.}
\label{fig:stemquadrant}
\end{figure}

\begin{prop}\label{thm:StemQuadrant}
Let $\H$ be a crooked halfspace with vertex 
$$
p:= \Vertex(\H)\in\E,
$$
stem quadrant $\Quad(\H)\subset\E$, and translational semigroup
$\V(\H) \subset \V$. Then:
$$
\Quad(\H)\ =\ 
p + \V(\H).
$$
\end{prop}
\noindent
The calculations in the proof will show that the $\V(\H)$ has a particularly simple form:
\begin{cor}\label{cor:StemQuadrant}
Let $\vs\in\V$ be a unit-spacelike vector and $p\in\E$. Then
$\V\big( \H(\vs,p)\big)$ consists of nonnegative linear combinations of $\xm{\vs}$ and $-\xp{\vs}$.
\end{cor}
\begin{proof}
Write the stem quadrant in the usual coordinates: 
$$
\Quad(\H) = \overline{(0,+,-)} = \{ c \le 0 = a \le b \}.
$$
A vector $\vv = (\alpha,\beta,\gamma)\in\V$ satisfies
$p + \vv\in\Quad(\H)$  if and only if:
\begin{equation}\label{eq:vInStemQuadrant}
\gamma \le 0 = \alpha \le \beta.
\end{equation}

We first show that if $p + \vv\in\Quad(\H)$, then $\vv\in\V(\H)$.
Suppose the coordinates $\alpha, \beta,\gamma$ of $\vv$ satisfy
\eqref{eq:vInStemQuadrant} and let $p = (a,b,c)\in\H$.
\begin{itemize}
\item If $a > 0$, then $\alpha + a = a > 0$ and $\beta + b > 0$.
\item If $a = 0$, then $\alpha + a = 0$ and $\beta + b > 0$, as well as
$\gamma + c < 0$.
\item If $a < 0$, then $\alpha + a = a < 0$ and $\gamma + c < 0$.
\end{itemize}
Thus $p +\vv \in \H$ as desired.

Conversely, suppose that $\vv\in\V(\H)$. 
Suppose that $\alpha > 0$. Choose $a  < -\alpha$ and $b < -\vert\beta\vert$.
Then $p=(a,b,c) \in \H$ but $\vv + p \notin\H$, a contradiction.
If $\alpha < 0$, then taking $a > -\alpha$ and $c > \vert\gamma\vert$ leads
to a contradiction. Thus $\alpha = 0$.

We next prove that $\beta\le 0$.
Otherwise $\beta > 0$ and taking $p=(0,b,c)$ where $c < -\gamma$ and $b > 0$
yields a contradiction. Similarly $\beta \ge 0$.
Thus  \eqref{eq:vInStemQuadrant} holds, proving $p + \vv \in\Quad(\H)$  as desired.
\end{proof}

\begin{prop} \label{prop:complement}
Let $\vs\in\V$ be a unit-spacelike vector and $p\in\E$. 
Then the complementary  open halfspace 
$\cH(p,\vs)^c$ equals $\H(p,-\vs)$ and
$$
\Quad\big(\cH(p,\vs)^c\big) = -\Quad\big(\H(p,\vs)\big).
$$
\end{prop}

\subsection{Linearization of crooked halfspaces}

Recall that in \S\ref{sec:hyperbolichalfplanes} we associated every spacelike vector in $\V$ to a halfplane in $\Ht$.  Given $\vs\in\V$ spacelike 
and $p\in\E$, we define the {\em linearization} of $\H(p,\vs)$ to be:
\begin{equation*}
\L(\H(p,\vs))=\h(\vs).
\end{equation*}



\noindent
We first show that linearization commutes with complement:
\begin{cor}\label{cor:complement} 
The correspondence $\L$ respects the involution:\newline
Suppose $\H\subset\E$ is a crooked halfspace with complementary 
halfspace $\H^c$. Then the linearization $\L(\H^c)$ is the halfplane
in $\Ht$ complementary to $\L(\H)$. 
\end{cor}
\begin{proof}
The spine reflection $R$ defined in \eqref{eq:spine}, \S\ref{sec:spine}
interchanges $\H$ and $\H^c$. Furthermore its linearization $\L(R)$ is
a reflection in $\partial\L(\H)$ which interchanges the particles in 
$\H$ and $\H^c$.  Therefore $\L(\H^c)$ and $\L(\H)$ are complementary
halfplanes in $\Ht$ as claimed.
\end{proof}

\noindent
Next we deduce that linearization preserves the relation of inclusion
of halfspaces.
\begin{cor}\label{cor:order-preserving}
The correspondence $\L$ respects the partial ordering:\newline
Suppose that $\H_1,\H_2\subset\E$ are crooked halfspaces,
with linearizations $\L(\H_1),\L(\H_2)\subset\Ht$. Then:
\[
\H_1 \subset \H_2\; \Longrightarrow \L(\H_1) \subset \L(\H_2).
\]
\end{cor}
\begin{proof}
Let $\vt\in\L(\H_1)$. Then there exists a particle $\ell$ parallel to $\vt$ 
such that $\ell\subset\H_1$. Since $\H_1\subset\H_2$, the particle
$\ell$ lies in $\H_2$. Thus $\vt\in\L(\H_2)$ as claimed.
\end{proof}
\noindent
Clearly $\L(\H_1) \subset \L(\H_2)$
does {\em not\/} in general imply that $\H_1\subset\H_2$.

\begin{cor}\label{cor:disjointness}
Suppose $\H_1,\H_2$ are disjoint crooked halfspaces.
Then their linearizations $\L(\H_1), \L(\H_2)$ are disjoint halfplanes in $\Ht$.
\end{cor}
\begin{proof}
Combine Corollary~\ref{cor:complement} and Corollary~\ref{cor:order-preserving}.
\end{proof}

\section{Symmetry}\label{sec:Symmetry}
In this section we determine various automorphism groups and endomorphism
semigroups of a crooked halfspace and the corresponding  orbit structure.

We begin by decomposing a halfspace into pieces, which will be invariant
under the affine transformations. From that we specialize to conformal
automorphisms, and finally isometries.

\subsection{Decomposing a halfspace}
The open halfspace $\H=\H(p, \vs)$  
naturally divides into three subsets,
the stem quadrant, defined by $a=0$ in null frame coordinates, 
and two {\em solid quadrants,\/}
defined by $a < 0$ and $a > 0$.
Recall that a
{\em solid quadrant\/} in a 
$3$-dimensional affine space is defined as the intersection of two ordinary
(parallel, that is, ``non-crooked'') halfspaces. 
Equivalently a solid quadrant is a connected component of the complement of the union
of two transverse planes in $\E$.

%

\subsection{Affine automorphisms}
We first determine the group of affine automorphisms of $\H$. 

First, every automorphism $g$ of $\H$ must fix 
$p := \Vertex(\H)$
and preserve the hinges $h_- $, $h_+$.
The crooked plane  $\partial\H$ is smooth except
along $h_- \cup h_+$
so $g$ leaves this set invariant.
Furthermore this set is singular only at the vertex $p = h_- \cap h_+$, 

Since $\H$ is vertexed at $p$, 
the affine automorphism $g$ must be linear 
(where $p$ is identified with the zero element of $\V$, of course).

The involution:
\begin{equation*} 
\rho  \left ( \begin{bmatrix} a \\ b\\ c \end{bmatrix} \right ) =  
\begin{bmatrix} -a \\ -c\\ -b \end{bmatrix}  
\end{equation*}
defined in \eqref{eq:rho}
preserves $\H$  (and also $\cH^c$), but interchanges $h_-$ and $h_+$. 
The involution preserves the particle:
$$
a = b+ c =  0.
$$
Thus, either $g$ or $g\rho$ will preserve $h_+$ and $h_-$.
We henceforth assume that $g$ preserves each hinge.

The complement of $h_-$ in $\partial\H$ has two components, 
one of which is smooth and the other singular (along $h_+$).
The smooth component is the wing $\W_-$, which must be preserved by $g$.
Thus $g$ preserves each wing.

Each wing lies in a unique (null) plane, and these two null planes intersect
in the spine defined in \eqref{eq:spine} in \S\ref{sec:spine}, 
the tachyon through $o$ parallel to $\vs$. The spine 
is also preserved by $g$. Thus $g$ is represented
by a linear map preserving the coordinate lines for the null frame
$(\vs,\xm{\vs},\xp{\vs})$, and therefore represented by a diagonal matrix.
We have proved:

\begin{prop}\label{prop:AffineAutos}
The affine automorphism group of $\H$ equals the double extension
of the group of positive diagonal matrices by the order two cyclic group 
$\langle\rho\rangle$. It is the image of the embedding:
\begin{align*}
\R^3 \rtimes (\Z/2) &\xrightarrow{\cong} \Aff(\E) \\
\big((s,t,u), \epsilon\big) &\longmapsto
\rho^{\epsilon} 
e^s \bmatrix e^u & 0 & 0 \\ 0 & e^t & 0 \\   0 & 0 & e^{-t} \endbmatrix.
\end{align*}
 where $\epsilon \equiv 0,1 \mod{2}$.
\end{prop}

\begin{figure}[ht]
\centerline{\includegraphics{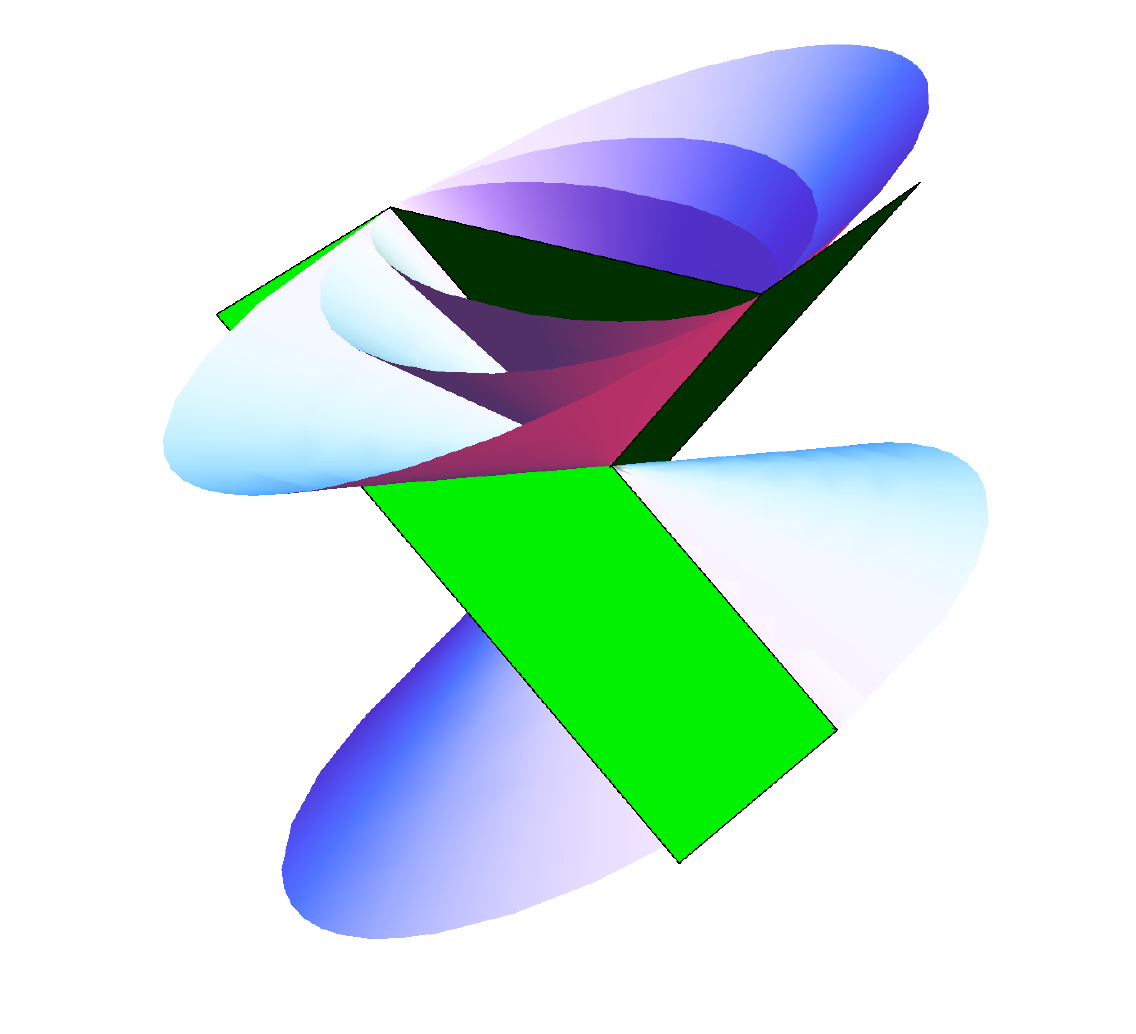}}
\caption{Affine automorphisms. This figure depicts a single crooked plane $C$,
and the lightcones for three affinely equivalent Lorentz structures in which 
$C$ is defined. For each of these Lorentz structures, a crooked 
halfspace complementary to $C$ meets the future in a region defining a halfplane in $\Ht$.}
\label{fig:AffineAutos}
\end{figure}

\subsection{Conformal automorphisms and isometries}
Lorentz isometries and homotheties generate the group
of conformal automorphisms of $\E$, that is the set of
{\em Lorentz similarity transformations\/}.
By \S\ref{sec:transformations} a conformal transformation 
(respectively isometry) is an affine automorphism $g$ whose linear part
$\L(g)$ lies in $\SOto \times \Rplus$ (respectively $\SOto$). 
By Proposition~\ref{prop:AffineAutos}, the linear part $\L(g)$ is a diagonal
matrix (in the null frame) so it suffices to check which diagonal matrices
act conformally (respectively isometrically). 


\begin{prop}\label{prop:ConfAutos}
Let $\H$ be a crooked halfspace.
\begin{itemize}
\item The group $\Conf(\H)$ of conformal automorphisms of $\H$ 
equals the double extension by the order two cyclic group  $\langle\rho\rangle$
of the subgroup of positive diagonal matrices generated by positive homotheties
and the one-parameter subgroup $\{\eta_t\mid t\in\R\}$ of boosts.
It is the image of the embedding:
\begin{align*}
\R^2 \rtimes (\Z/2) &\xrightarrow{\cong} \Conf(\H) \\
\big((s,t), \epsilon\big) &\longmapsto
\rho^{\epsilon}  e^s \eta_t 
\end{align*}
where $\epsilon \equiv 0,1 \mod{2}$.
\item The isometry group of $\H$ equals the double extension
of the one-parameter subgroup $\{\eta_t\mid t\in\R\}$ of boosts,
by the order two cyclic group $\langle\rho\rangle$.
\end{itemize}
\end{prop}

\subsection{Orbit structure}

In this section we describe the orbit space of $\H$ under the action
of its conformal automorphism group $\Conf(\H)$. The main goal
is that the action is proper with orbit space homeomorphic to a half-closed
interval. The function:
$$
\Phi(a,b,c) := b c /a^2
$$
defines a homeomorphism of the orbit space with $\R\ \cup\{-\infty\}$.
The action is not free.  
The only fixed points are rays in the stem quadrant
which are fixed under conjugates of the involution $\rho$.

\subsubsection{Action on the stem quadrant}
\begin{lemma}\label{lem:propernessStem}
The identity component $\Confo(\H)\cong\R^2$ acts transitively  and freely on the stem quadrant $\Quad(\H)$. 
\end{lemma}
\begin{proof}
Fix a basepoint $q_0$ in the stem quadrant:
\begin{equation}\label{eq:BasepointInQ}
q_0 := \bmatrix 0 \\ 1 \\ -1 \endbmatrix
\end{equation}
Then:
$$
\eta_s \xi_t (q_0) = \bmatrix 0 \\ e^{s+t} \\ -e^{s-t} \endbmatrix
$$
An arbitrary point in the stem quadrant $\Quad(\H)$ is: 
\begin{equation}\label{eq:ArbitraryPoint}
p  = \bmatrix a \\ b \\ c \endbmatrix, 
\end{equation}
with $a=0$ and $b \geq 0 \geq c$.
Then:
\begin{align*}
p &= \sqrt{- b c}  \bmatrix 0 \\ \sqrt{-b/c} \\ -\sqrt{-c/b} \endbmatrix \\
& = \eta_s \xi_t (q_0),
\end{align*}
where:
\begin{align*}
  s &= \frac{\log(b) + \log(-c)}2, \\
  t &= \frac{\log(b) - \log(-c)}2
\end{align*}
are uniquely determined. 
\end{proof}
However, the group of similarities $\Conf(\H)\ =\ 
\Confo(\H)\rtimes \langle\rho\rangle$
does {\em not\/} act freely on the stem quadrant as the involution $\rho$ fixes the ray
\begin{equation*}
\Fix(\rho) := \left\{ \begin{bmatrix} 0\\ b\\ - b\end{bmatrix}  \   
{\Bigg|} \  b >  0\right\}.
\end{equation*}

\subsubsection{Action on the solid quadrants}

\begin{figure}[ht]
\centerline{\includegraphics{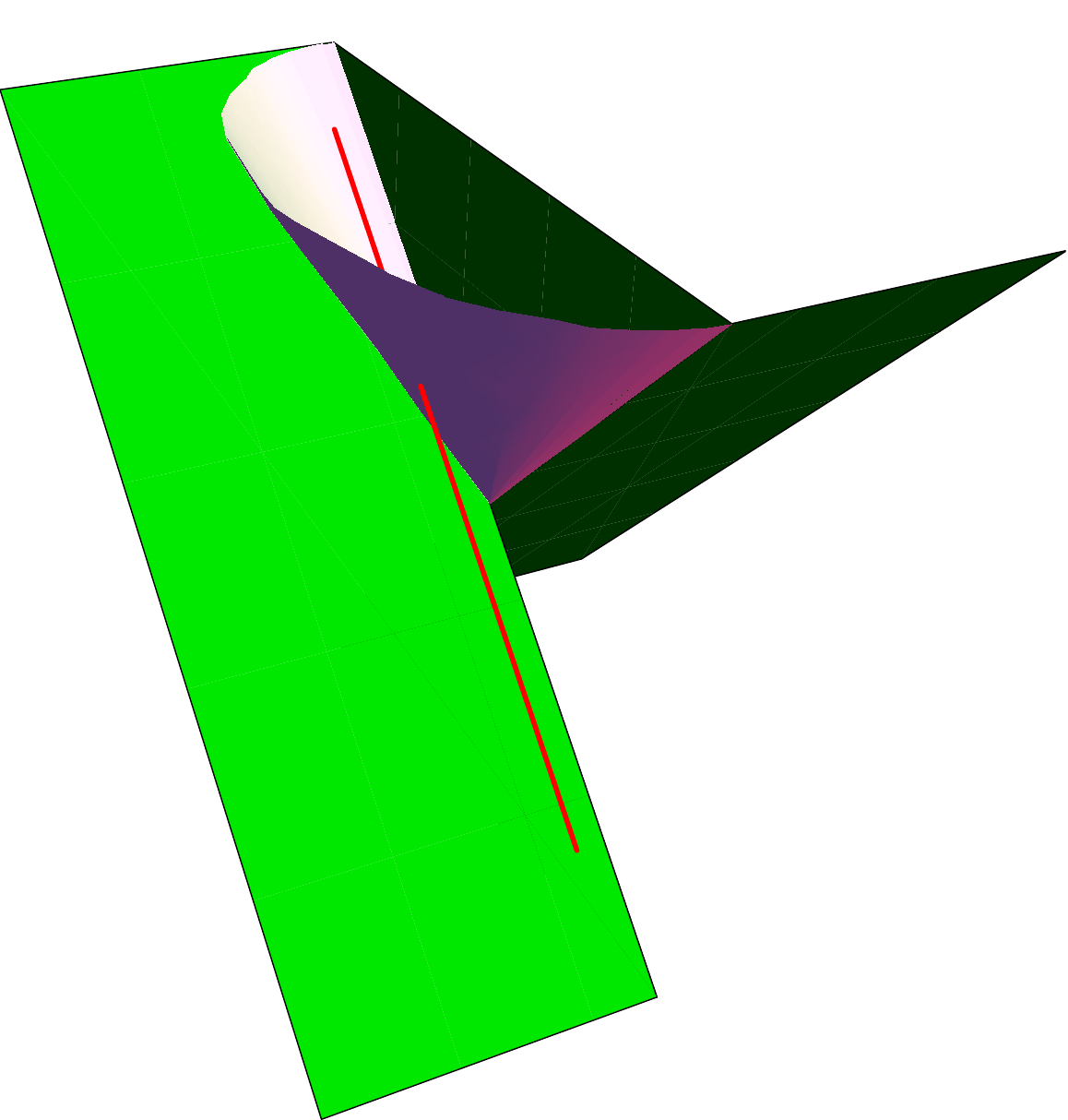}}
\caption{A slice of a basic orbit. The 
line is transverse to all of the orbits.}
\label{fig:orbitslice}
\end{figure}

The stem quadrant $\Quad(\H)$ divides $\H$ into two solid quadrants, $(+,+,\pm)$  defined by $a>0$, and $(-,\pm,-)$ defined by $a<0$.

\begin{lemma}\label{lem:propernessSolid}
The identity component $\Confo(\H)$ acts properly and freely on each 
solid quadrant in $\H\setminus\Quad(\H)$, and $\rho$ interchanges them.
The function:
\begin{align*}
\H \setminus \Quad(\H)&\xrightarrow{\Phi} \R\\\
\bmatrix a \\ b \\c \endbmatrix &\longmapsto bc/a^2
\end{align*}
defines a diffeomorphism:
$$
\big(\H\setminus\Quad(\H)\big)/\Conf(\H) \xrightarrow{\approx} \R.
$$ 
\end{lemma}
\begin{proof}
We only consider the solid quadrant $(+,+,\pm)$, 
since $(-,\pm,-)$ follows from this case by applying $\rho$.

We show that the set $B$ of all:
\begin{equation}\label{eq:pgamma}
p_\gamma := \bmatrix 1 \\ 1 \\ \gamma \endbmatrix,
\end{equation}
where $\gamma\in\R$, is a slice for the action on $(+,+,\pm)$.
Namely, the map:
\begin{align*}
\Confo(\H) \times B &\longrightarrow  (+,+,\pm)\subset\H\\
\big( (\eta_s,\xi_t), p_\beta\big) &\longmapsto  \eta_s \xi_t ( p_\beta)
\end{align*}
is a diffeomorphism.  If $p$ is an arbitrary point as in \eqref{eq:ArbitraryPoint}
above, and $a,c > 0$, then:
\begin{align*}
s & := \log(a) \\
t  & := \log(b/a) \\
\beta & := \Phi(p) = bc/a^2
\end{align*}
uniquely solves:
$$
\eta_s \xi_t (p_\beta) = p
$$
and defines the smooth inverse map.
Thus $\Confo(\H)$ acts properly and freely on each solid quadrant.
Since $\rho$ interchanges these quadrants, 
$\Conf(\H) = \Confo(\H) \rtimes \langle\rho\rangle$  acts properly and freely on 
$\H\setminus\Quad(\H)$ as claimed and $\Phi$ defines a quotient map.
\end{proof}

\subsubsection{Putting it all together}

Now combine Lemmas~\ref{lem:propernessStem} and \ref{lem:propernessSolid}
to prove that $\Conf(\H)$ acts properly on $\H$.
Furthermore, the quotient map $\Phi$ takes $\H$ 
onto the infinite half-closed interval $\{-\infty\} \cup \R$.

The main problem is that the slice $B$ used in the proof of
Lemma~\ref{lem:propernessSolid} does not extend to $\Quad(\H)$,
since $a \equiv 1$ on $B$ and $\Quad(\H)$ is defined by $a=0$.
To this end we replace the slice $B$, for parameter values $1\ge a > 0$,
by an equivalent slice $B'$. 
The new slice $B'$ is parametrized by a variable $1 \ge a\ge 0$, which converges to the basepoint $q_0\in\Quad(\H)$ as 
defined in \eqref{eq:BasepointInQ} as $a \nearrow 1$ and
converges to $p_0$ as $a \searrow 0$. 
These points $p_0$ on $B$ corresponding
to parameter value $\beta = 0$, and the basepoint on $\Quad(\H)$ equal:
$$
p_0 = \bmatrix 1 \\ 1 \\ 0 \endbmatrix, \quad
q_0 = \bmatrix 0 \\ 1 \\ -1 \endbmatrix,
$$
respectively. Thus we replace the segment of the slice $B$ for $\beta\le 0$,
by points of the form:
$$
p'_a := \bmatrix a \\ 1 \\ a - 1  \endbmatrix
$$
for $1\ge a > 0$. The corresponding $\gamma$-parameter is:
$$
\gamma(a) := \Phi(p'_a) = \frac{a-1}{a^2}
$$
with inverse function:
$$
a(\gamma) := \frac{1-\sqrt{1-4\gamma}}{2\gamma}
$$
We obtain inverse diffeomorphisms:
\[
(0,1) \xrightarrow{\gamma} (-\infty,0) 
\xrightarrow{a} (0,1) 
\]
which extend to homeomorphisms $[0,1] \approx [-\infty,0].$

\subsubsection{A global slice}
Thus we construct a slice $\sigma$ for the $\Confo(\H)$-action on $\H$ using the function $\gamma = \Phi(p)$ extended to:
$$
\H \xrightarrow{\Phi} \{-\infty\} \cup \R 
$$
by sending $\Quad(\H)$ to $-\infty$. Furthermore we can extend $\sigma$ uniquely
to a $\rho$-equivariant slice for the action of $\Confo(\H)$ on $\H$. 
We define the continuous slice for the parameter $-\infty  < a < \infty$;
it is smooth except for parameter values $a = -1, 0, 1$ where it equals:
$$
\sigma(-1) := \bmatrix -1 \\ 0 \\ -1 \endbmatrix, \quad
\sigma(0) := q_0  = \bmatrix 0 \\ 1 \\ -1 \endbmatrix, \quad
\sigma(1) :=  \bmatrix 1 \\ 1 \\ 0 \endbmatrix 
$$
On the intervals $(-\infty, -1]$, $[-1,0]$, $[0,1]$, and $[1,\infty)$, 
smoothly interpolate between these values:
\begin{alignat*}{2}
&\text{For~}   a \le -1, &  \quad \sigma(a) &:= 
\;\bmatrix -1 \\  -\gamma(-a) \\ -1 \endbmatrix \\
&\text{For}  -1 \le a \le 0, & \quad  \sigma(a) &:= 
\;\bmatrix a \\  a + 1 \\ -1 \endbmatrix \\
&\text{For~}  0\le a \le 1, &\quad   \sigma(a) &:= 
\;\bmatrix a \\ 1\\ a-1  \endbmatrix  \\
&\text{For~}  1\le a, &  \quad \sigma(a) &:= 
\;\bmatrix 1 \\ 1\\ \gamma(a)  \endbmatrix 
\end{alignat*}

\section{Lines in a halfspace}\label{sec:Particles}
In this section we classify the lines which lie entirely in a crooked halfspace.
The natural context in which to initiate this question is affine; 
we develop a criterion in terms of the stem quadrant for a line to lie
in a halfspace. 

\subsection{Affine lines}

Given an affine line $\ell\subset\E$ and a point $o\in\E$, 
a unique line, denoted $\ell_o$, is parallel to $\ell$ and contains $o$.
Let $\eta_t^{(o)}$ denote the one-parameter group of
homotheties fixing $o$
and preserving the crooked halfspace:
\begin{align*} 
\E &\xrightarrow{\eta_t^{(o)}} \E \\
o + \vv &\longmapsto  o + e^t \vv.
\end{align*}
Then: 
\begin{equation}\label{eq:HomothetiesLimit}
\ell_o = \lim_{t\to-\infty} \eta_t^{(o)}(\ell).
\end{equation}

\begin{lemma}
If $\ell\subset\H$ and $o = \Vertex(\H)$, then $\ell_o\subset\cH$.
\end{lemma}
\begin{proof}
The homotheties $\eta_t^{(o)}\in\Aff(\H)$ so 
$\eta_t^{(o)}(\ell)\subset\H$.
Apply \eqref{eq:HomothetiesLimit} to the 
closed set $\cH$ to conclude that
$\ell_o\subset\cH$.
\end{proof}
Now let $\Ss\subset\E$ be the stem plane of $\H$.
Unless $\ell$ is parallel to $\Ss$, it meets $\Ss$ in a unique
point $p$. Since $\ell\subset\H$ and: 
$$
\overline{\H \cap \Ss} = \Quad(\H),
$$ 
the stem quadrant $\Quad(\H)\ni p$. 
Then translation of $\ell_o$ by $p-o$ 
is a line through $p$ which is parallel to $\ell$, 
and thus equals $\ell$. Therefore:
\begin{lemma}
Every line contained in $\cH$ 
not parallel to $\Ss$ is the translate of a line in $\cH$
passing through $\Vertex(\H)$ by a vector in $\V(\H)$.
\end{lemma}

\subsubsection{Lines through the vertex}\label{sec:LinesThruVertex}
Now we determine when a line $\ell$ translated by a nonzero vector 
lies in $\cH$. Suppose that $\ell$ is spanned by the vector:
\begin{equation}\label{eq:vector}
\vv = \bmatrix \alpha \\ \beta \\ \gamma \endbmatrix.
\end{equation}
We shall use the {\em orthogonal projection\/} to the stem plane $\Ss$
defined by:
\begin{align*}
\E & \longrightarrow \Ss \subset \E \\
\bmatrix x \\ y \\ z \endbmatrix &\longmapsto 
\bmatrix 0 \\ y \\ z \endbmatrix.
\end{align*}

First suppose that $\ell$ doesn't lie in the
stem plane $\Ss$, that is, $\alpha\neq 0$. 
By scaling, assume that $\alpha = 1$.
Then $\ell$ consists of all vectors:
$$
a\vv = \bmatrix a \\ a\beta \\ a\gamma \endbmatrix,
$$
where $a\in\R$. Since:
$\H = (+,+,\pm) \cup (0,+,-)\cup (-,\pm,-)$, the condition
that $a\vv\in\cH$ is equivalent to the two conditions:
\begin{itemize}
\item $a<0$ implies $a\gamma \le 0$;
\item $a>0$ implies $a\beta \ge 0$.
\end{itemize}
Thus $\ell\subset\cH$ if and only if $\beta, \gamma \ge 0$.
Moreover $\ell\cap \cH = \{o\}$ if and only if $\beta, \gamma <  0$.

It remains to consider the case when $\ell\subset\Ss$, that is, $\alpha = 0$.
Since: 
$$
\Stem(\H) \  =\  \overline{(0,+,+)\cup (0,-,-)},
$$ 
the condition that $\ell\subset\cH$ is equivalent to the conditions 
$\beta \gamma \ge 0$, that is, $\ell$ lies in a solid quadrant in $\cH$
whose projection to $\Ss$ maps to $\overline{\Stem(\H)}\subset \Ss$.
We have proved:

\begin{lemma}
Suppose  $\ell\subset \cH$ is a line.
Then orthogonal projection to $\Ss$ 
maps $\ell$ to $\overline{\Stem(\H)}\subset \Ss$.
\end{lemma}

\subsection{Lines contained in a halfspace and linearization}

Suppose that $\vv$ defined in \eqref{eq:vector} is future-pointing timelike, 
and that $\R\vv\subset\cH$. Then 
the discussion in \S\ref{sec:LinesThruVertex} implies that
necessarily $\beta,\gamma\ge 0$. 
Now $\alpha^2 < \beta\gamma$ implies that 
$\beta,\gamma > 0$, that is, that all coordinates $\alpha, \beta, \gamma$
have the same (nonzero) sign.\
This condition is equivalent to the future-pointing timelike vector having
positive inner product with the unit-spacelike vector:
$$
\vs_0 := \bmatrix 1 \\ 0 \\ 0 \endbmatrix,
$$
and this condition defines a {\em halfplane\/} in $\Ht$.
We conclude:

\begin{thm}
Let $\H(p,\vs)\subset\E$ be a crooked halfspace.
Then the collection of all future-pointing unit-timelike vectors parallel
to a particle contained in $\H(p,\vs)$
is the halfplane $\h(\vs)\subset\Ht$.
\end{thm}


\subsection{Unions of particles}

We close this section with a converse statement.

\begin{thm}
Let $\H\subset\E$ be a crooked halfspace. 
Then every $p\in\H$ lies on a particle contained in $\H$.
\end{thm}
\begin{proof}
It suffices to prove the theorem for the
halfspace:
\[
\H = (+, +, \pm) \cup (0, +,-) \cup (-, \pm -).
 \]
 Start with any $p$ 
in the open  solid quadrant $(+, +, \pm) $ 
We will now describe a point $q$ 
in the stem quadrant $(0, +,-) $ 
for which the vector $p-q$ is timelike.  
Write:
$$
p = \left( \begin{array}{c} a \\ b\\ c  \end{array}\right), \ 
q = \left( \begin{array}{c} 0 \\ B \\ C  \end{array}\right) 
$$
so that $a, b >0$ and 
$B>0 >  C$.
First choose $B$ so that $0<B<b$, then
choose:  $$
C < \mbox{min}\bigg(0, c - \frac{a^2}{b-B}\bigg).
$$ 
Let: $$
\vv = p-q = \begin{bmatrix} a \\ b-B \\ c-C \end{bmatrix} 
$$ 
so that:
$$
\vv\cdot\vv  = a^2 - (b-B)(c-C) < a^2 -a^2 =0 , 
$$ 
$b-B >0$ and $c-C > 0$. 
That is, the line $q+ t\vv$  is a particle. All of the points on the line where $t>0$ lie inside the solid quadrant $( +, +,\pm )$ and 
all of the points where $t<0$ lie inside the solid quadrant $(-, \pm, -)$.


A similar calculation applies to points in the $(-,\pm , -)$ solid quadrant.

It remains only to consider points $p\in\Quad(\H) = (0, +,-) $. Any timelike vector pointing inside of $\H$ will suffice, but 
choose the timelike vector:
$$
\vv= \begin{bmatrix} 1/2 \\ 1\\ 1 \end{bmatrix}.
$$
Consider the line:
$$
p + t \vv = \left( \begin{array}{c} t/2  \\ b + t \\ c+ t \end{array} \right) . 
$$ 
All points where $t>0$ lie inside $(+,+, \pm)$, and all points where $t<0$ lie inside $(-, \pm , -)$. 

\end{proof}

\section{Disjointness criteria}
In this section we revisit the theory developed in~\cite{DrummGoldman} in terms of the notion of stem quadrants and
crooked halfspaces. If $\H_1, \H_2$ are disjoint crooked halfspaces,
then their linearizations $\h_i = \L(\H_i)$ are disjoint halfplanes in $\Ht$
(Corollary~\ref{cor:disjointness}). Suppose that $\vs_i$ is the spacelike
vector corresponding to $\h_i$ as in \S\ref{sec:hyperbolichalfplanes}
and that they are consistently oriented. 
\begin{defn}
Let $\vs_1,\vs_2$ be consistently oriented spacelike vectors.  The interior of $\V(\vs_1)-\V(\vs_2)$ is called the cone of \emph{allowable translations}, denoted  $\A(\vs_1, \vs_2)$.
\end{defn}
We show that two (open) crooked halfspaces with disjoint linearizations are disjoint
if and only if the vector between their vertices lies in the closure of the cone of allowable translations.

\begin{thm}\label{thm:disjointness}
Suppose that $\vs_i$ are consistently oriented unit-spacelike vectors and
$p_1,p_2\in\E$.
Then the closed crooked halfspaces 
$\cH(p_1, \vs_1)$ and $\cH(p_2, \vs_2)$ are disjoint if and only if: 
\begin{equation}\label{eq:disjointnesscondition}
p_1 - p_2 \in \A(\vs_1, \vs_2).
\end{equation}
Similarly $\H(p_1, \vs_1)\cap \H(p_2, \vs_2) = \emptyset$ if and only if 
$p_1 - p_2$ lies in the closure of $\A(\vs_1, \vs_2)$.
\end{thm}
\begin{proof}
We first show that \eqref{eq:disjointnesscondition} implies
that $\cH(p_1, \vs_1) \cap \cH(p_2, \vs_2) = \emptyset$.
Choose $\vv_i\in \V(\vs_i)$ for $i=1,2$ respectively.
Choose an arbitrary origin $p_0\in\E$ and let $p_i := p_0 + \vv_i$.

Lemma~\ref{lem:consistentorientation} implies that
the crooked halfspaces $\H(p_0, \vs_1 )$ and 
$\H(p_0, \vs_2)$ are disjoint.
By Theorem~\ref{thm:StemQuadrant}, 
$$
\H(p_i, \vs_i) := \H(p_0, \vs_i)  + \vv_i \subset \H_i
$$ 
Thus $\H(p_1, \vs_1)$ and $\H(p_2, \vs_2)$ 
are disjoint.

Conversely, suppose that $\H(p_1, \vs_1) \cap \H(p_2, \vs_2 ) = \emptyset$. 
We use the following results from \cite{DrummGoldman}, 
(Theorem~6.2.1 and Theorem~6.4.1),
which are proved using a case-by-case analysis of intersections of wings and stems:
\begin{prop}
Let $\vs_i\in\V$ be consistently oriented unit-spacelike vectors and 
$p_i\in\E$, 
for $i=1,2$. Then $\CP(p_1, \vs_1)\cap \CP(p_2, \vs_2) = \emptyset$
if and only if:
\begin{itemize}
\item
for ultraparallel $\vs_1$ and $\vs_2$, 
\begin{equation}\label{eq:drummgoldmanu}
\ldot{(p_2-p_1)}{(\lcross{\vs_1}{\vs_2})}  > 
\vert\ldot{(p_2-p_1)}{\vs_1}\vert +
\vert\ldot{(p_2-p_1)}{\vs_2}\vert , 
\end{equation}
\item
for asymptotic $\vs_1$ and $\vs_2$ (where $\xm{\vs_1} = \xp{\vs_2}$),
then: 
\begin{align}\label{eq:drummgoldmana}
\ldot{(p_2-p_1)}{\vs_1} & < 0, \notag\\ 
\ldot{(p_2-p_1)}{\vs_2} & < 0, \notag\\ 
\ldot{(p_2-p_1)}{
(\lcross{\xp{\vs_1}}{\xm{\vs_2}})
} & >  0 .
\end{align}
\end{itemize}
\end{prop}

First suppose that $\vs_1$ and $\vs_2$ are ultraparallel and consider \eqref{eq:drummgoldmanu}. The inequality 
defines an infinite pyramid whose sides are defined where the absolute values in  \eqref{eq:drummgoldmanu} arise from
multiplication of $\pm 1$.

Corollary~\ref{cor:StemQuadrant} implies that  $\A(\vs_1, \vs_2)$
consists of all positive linear combinations of:
$$
\xm{\vs_1},  
-\xp{\vs_1},  
-\xm{\vs_2},  
\xp{\vs_2}.  
$$
Each of these vectors defines one of the four corners of the infinite pyramid. We show this for two  vectors, while the other
two vectors follow similar reasoning.

Set  $p_2 - p_1 = \xm{\vs_1}$, and plug this value into both sides of \eqref{eq:drummgoldmanu}.
The left-hand side expression, using \eqref{eq:cross} and \eqref{eq:PlusMinus}, is:
$$
\ldot{\xm{\vs_1}}{(\lcross{\vs_1}{\vs_2})} = 
\Det (\xm{\vs_1}, \vs_1, \vs_2 ) = 
- \ldot{\vs_2}{(\lcross{\vs_1}{\xm{\vs_1}}) =
  \ldot{\vs_2}}{\xm{\vs_1}}.
$$
By the definition of consistent orientation, this term is positive.
The right-hand side expression is:
$$ 
\vert\ldot{\xm{\vs_1}}{\vs_1}\vert +
\vert\ldot{\xm{\vs_1}}{\vs_2}\vert  = \vert\ldot{\xm{\vs_1}}{\vs_2}\vert  .
$$
Thus, the vector  $p_2 - p_1 = \xm{\vs_1}$ defines the ray on the corner with the sides defined by 
$ \ldot{\vs_2}{\xm{\vs_1}}=\vert\ldot{\xm{\vs_1}}{\vs_2}\vert$.

Now, set  $p_2 - p_1 = -\xp{\vs_1}$, and plug this value into both sides of \eqref{eq:drummgoldmanu}.
The left-hand side expression, using \eqref{eq:cross} and \eqref{eq:PlusMinus}, is:
$$
\ldot{-\xp{\vs_1}}{(\lcross{\vs_1}{\vs_2})} =  
-\Det (\xp{\vs_1}, \vs_1, \vs_2 )  =
 \ldot{\vs_2}{(\lcross{\vs_1}{\xp{\vs_1}}) =   \ldot{\vs_2}}{\xp{\vs_1}}.
$$
By the definition of consistent orientation, this term is positive.
The right-hand side expression is:
$$ 
\vert\ldot{-\xp{\vs_1}}{\vs_1}\vert +
\vert\ldot{-\xpm{\vs_1}}{\vs_2}\vert  = \vert\ldot{\xm{\vs_2}}{\vs_2}\vert  .
$$
Thus, the vector  $p_2 - p_1 = -\xm{\vs_1}$ defines the ray on the corner with the sides defined by  
$  \ldot{\vs_2}{\xm{\vs_1}}=\vert\ldot{\xm{\vs_1}}{\vs_2}\vert$.

The asymptotic case \eqref{eq:drummgoldmana} is similar. 
The set of allowable translations, defined by \eqref{eq:drummgoldmana}, 
has three faces whose bounding rays  are parallel  to:
\[
\xm{\vs_2},\  -\xp{\vs_2}=-\xm{\vs_1},\  \xp{\vs_1}.
\]
The rest of the proof is analogous to the ultraparallel case
\eqref{eq:drummgoldmanu}.


\end{proof}

\section{Crooked foliations}
In this final section we apply the preceding theory to 
foliations of $\E$ by crooked planes.
These foliations linearize to foliations of $\Ht$ by geodesics.
Thus we regard crooked foliations as  affine deformations of geodesic
foliations of $\Ht$.  In this paper we consider affine deformations
of the foliation $\Fl$ of $\Ht$ by geodesics orthogonal to a fixed
geodesic $\ell\subset\Ht$.

\subsection{Foliations}

Let $M^m$ be an $m$-dimensional  topological manifold.
For $0 \le q \le m$, denote the coordinate projection by:
$$
\R^m \xrightarrow{\Pi} \R^q.
$$

\begin{defn}
A \emph{foliation} of codimension $q$  of 
 $M^m$ is a decomposition of $M$ into codimension $q$ submanifolds $F_x$, called \emph{leaves}, (indexed
by $x\in M$) 
together with an atlas of coordinate charts (homeomorphisms):
\[
U\xrightarrow{\psi_U} \R^m 
\]
such that the inverse images $(\psi_U)^{-1} (y)$, for $y\in\R^q$ are the intersections
$U\cap L_x$.
A \emph{crooked foliation}  of an open subset $\Omega\subset \E$ 
is a foliation of $\Omega$ by piecewise-linear leaves $F_x$ which are intersections of $\Omega$
with crooked planes.
More generally, if $\Omega$ is an open subset such that $\bar{\Omega}$ is a codimension-$0$ submanifold-with-boundary, we require that 
$\partial\Omega$
has a coordinate atlas with charts mapping to open subsets of crooked planes.
\end{defn}
The linearization $\L(F_x)$ of each leaf $F_x$ is a geodesic in $\Ht$,
and these geodesics foliate on open subset of $\Ht$. 
Thus the {\em linearization\/} of a crooked foliation is a geodesic foliation
of an open subset of $\Ht$.

\subsection{Affine deformations of orthogonal geodesic foliations } 
Given a geodesic $\ell\subset\Ht$,  the geodesics perpendicular to $\ell$ foliate $\Ht$. 
This geodesic foliation of $\Ht$, denoted $\Fl$, linearizes crooked foliations 
$\mathcal{F}$ as follows.

The leaves of $\Fl$, form a path of geodesics parametrized by a path of unit-spacelike vectors $\vs_t$.
The leaves of a crooked foliation $\mathcal{F}$ with linearization $\L(\mathcal{F}) = \Fl$ 
are crooked planes with directors $\vs_t$. Thus $\mathcal{F}$ will be specified by a path
$p_t$ in $\E$ such that the leaves of $\mathcal{F}$ are crooked planes $\CP(\vs_t,p_t)$.
We call $p_t$ the {\em vertex path\/} of $\mathcal{F}$.

\begin{prop}\label{prop:vertexpath}
Let $p_t$, $a\leq t\leq b$ be a 
regular path in $\E$ such that $p_t'$ belongs to the interior of the translational semigroup $\V(\H(\vs_t))$.  
Then for every $a\leq t_1,t_2\leq b$, the crooked planes $\CP(\vs_{t_1},p_{t_1})$ and $\CP(\vs_{t_2},p_{t_2})$ are disjoint.
\end{prop}

\begin{proof}
We may assume the foliations gives rise to an ordering of the corresponding halfspaces:
 $$
\H(\vs_t)\subset\H(\vs_a),
$$ 
for all $a\leq t\leq b$. Let $t_1<t_2$, and set $\vv_t=p_t'$:
\begin{equation*}
p_{t_2}-p_{t_1}=\int_{t_1}^{t_2}~\vv_t dt
\end{equation*}
Since $\vs_t$ is a continuous path, $\h(\vs_{t_1})$ lies between $\h(\vs_{a})$  and $\h(\vs_{t_2})$, which 
lies between $\h(\vs_{t_1})$ and $\h(\vs_{b})$.  In particular, for every $t\in (t_1,t_2)$:
\begin{equation*}
\V(\H(\vs_t))\subset\A(\vs_{t_2},-\vs_{t_1})
\end{equation*}
(Note that $\vs_{t_1},~\vs_{t_2}$ are not consistently oriented but $-\vs_{t_1},~\vs_{t_2}$  are.)
Thus every $\vv_t$ belongs to $\A(\vs_{t_2},-\vs_{t_1})$.  Since $\A(\vs_{t_2},-\vs_{t_1})$ is a cone, it follows 
that $p_2-p_1$ belongs to it as well.
\end{proof}

We explicitly calculate a family of examples.
The spacelike vectors perpendicular to the geodesic $\ell\subset\Ht$
defined by the vector $\vs$ in \eqref{eq:100vector} form the path
\[
\vs_t := \bmatrix 0 \\ \cosh(t) \\ \sinh(t) \endbmatrix \in \V.
\]
\noindent
As in \S \ref{sec:NullFrames}, the corresponding null vectors are:
\[
\vs_t^\pm = \frac1{\sqrt{2}}
\bmatrix \mp 1 \\ \sinh(t) \\ \cosh(t) \endbmatrix \in \V
\]
which, for any $p\in\E$ define crooked halfspaces
$\H(p,\vs_t)$. 
The translational semigroups $\V\big(\H(p,\vs_t)\big)$ consist of all
\[
\vv(t) := a_t \vs_t^- - b_t \vs_t^+ =
\frac1{\sqrt{2}}
\bmatrix -\big(a_t  + b_t\big) \\  \big(a_t  - b_t\big) \sinh(t)\\  \big(a_t  - b_t\big) \cosh(t)
\endbmatrix
\]
where $a_t,b_t>0$. 
By Lemma~\ref{prop:vertexpath}, the vertex path $p_t$ is obtained by integrating $\vv(t)$.

The linearization $\Fl$ of the crooked foliations with leaves $\CP(p_t,\vs_t)$ is invariant under the one-parameter
group of transvections $\xi_t$ defined in \eqref{eq:Transvections} in \S\ref{sec:Tbhr}.
When the coefficients $a_t, b_t$ are positive constants, then the leaves $L_t$ are orbits $\gamma_t(L_0)$ of a fixed
leaf $L_0$ by an affine deformation $\gamma_t$ of $\xi_t$.
For example, let   $a_t = \sqrt{2} a$, $b_t = \sqrt{2} b$, where $a,b>0$. 
Then the vertex path is:
\[
p_t = \bmatrix 
-(a+b) t \\ (a-b) \cosh(t) \\ (a-b) \sinh(t) \endbmatrix.
\]
When $a=b$,  the vertex path is a spacelike geodesic.
Figure~\ref{fig:foliation} depicts such a foliation.

\begin{figure}[htb]
\centerline{\includegraphics[scale=.56]{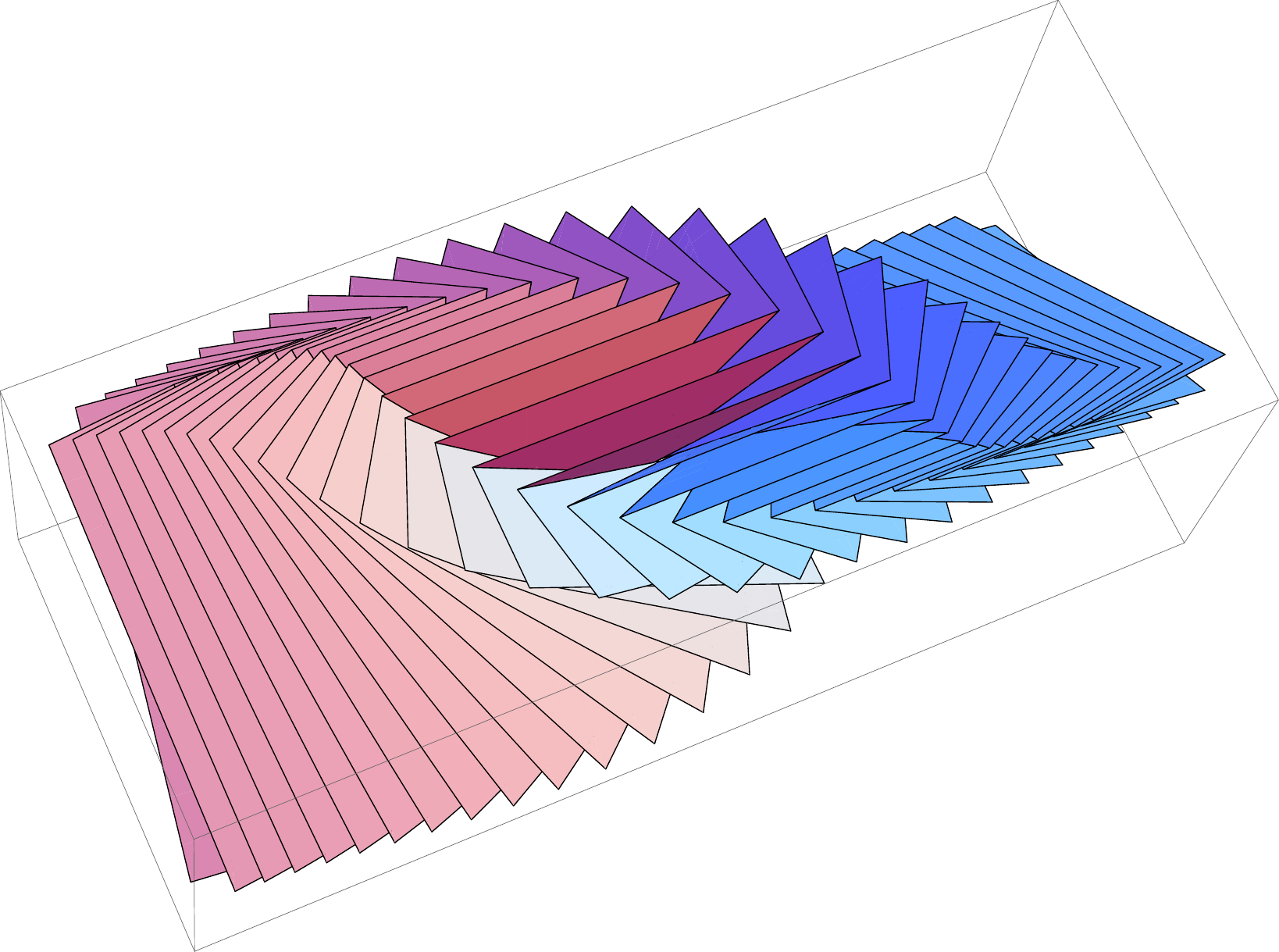}}
\caption{A crooked foliation}
\label{fig:foliation}
\end{figure}

\begin{figure}\label{fig:ZZ}
\begin{subfigure}[b]{0.6\textwidth}
\includegraphics[scale=.8]{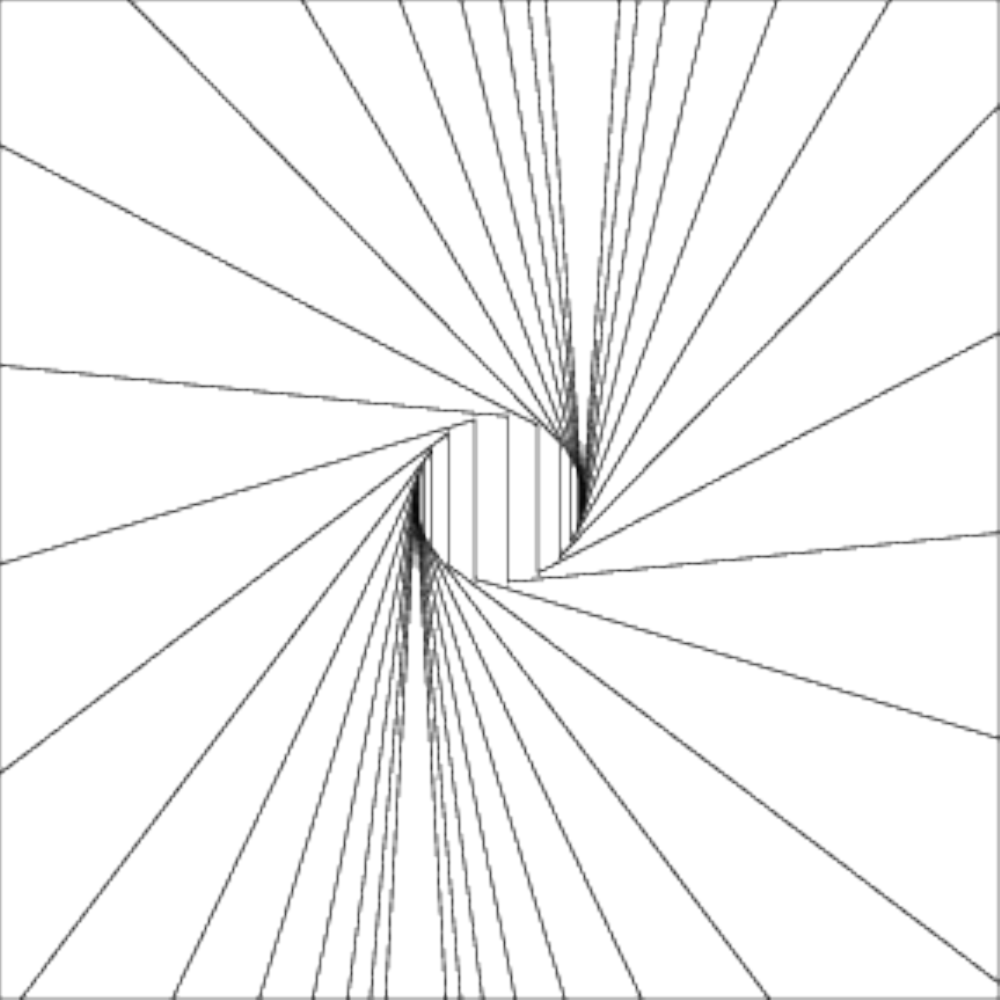}
\end{subfigure}
\begin{subfigure}[b]{0.6\textwidth}
\includegraphics[scale=0.8]{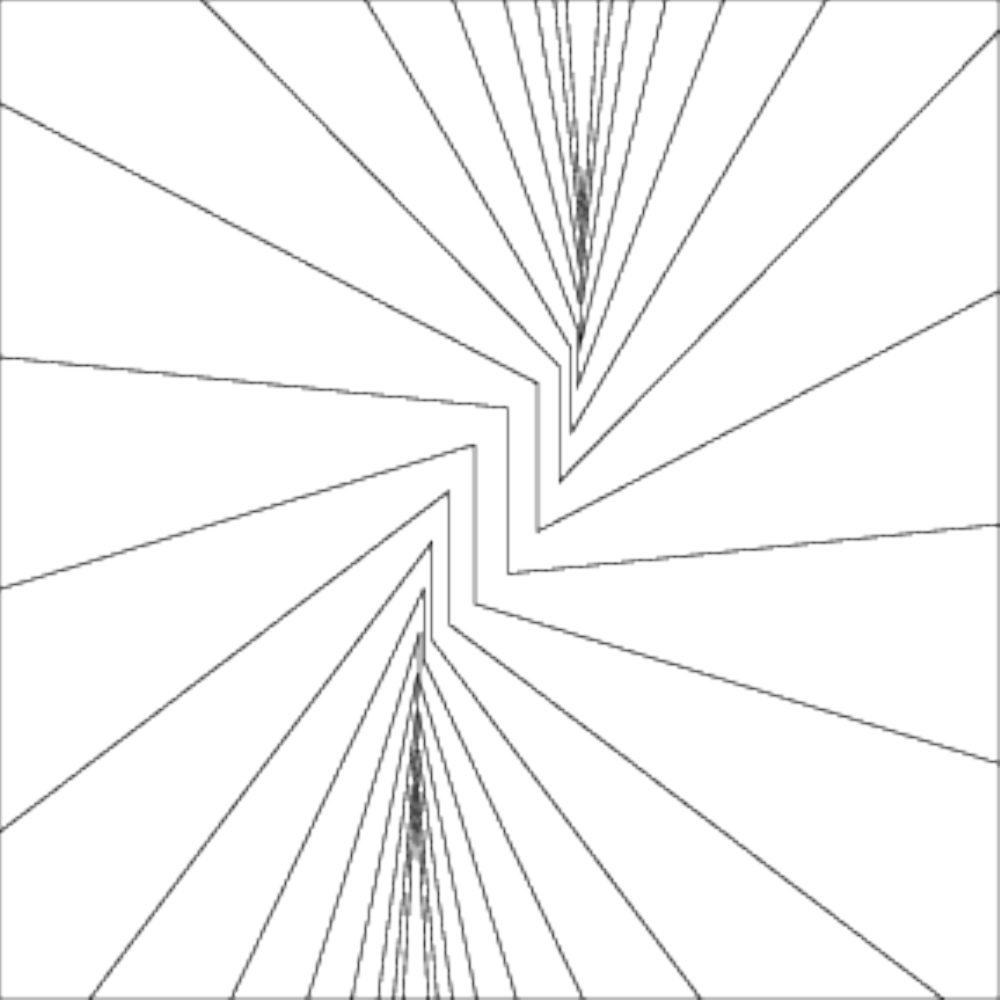}
\end{subfigure}
\caption{
One can visualize crooked planes by their intersections with a fixed definite plane,
called {\em zigzags\/} in \cite{CG}, \S 3.3.
The first picture illustrates the family of zigzags arising
as intersections of the crooked planes $\CP(\vs_t,0)$
vertexed at the origin $0$.
The second picture illustrates the family of zigzags 
arising from a crooked foliation where $p_t$ is a spacelike geodesic.
}
\end{figure}

\begin{figure}
  \centering
  \begin{subfigure}[b]{0.6\textwidth}
     \includegraphics[scale=1]{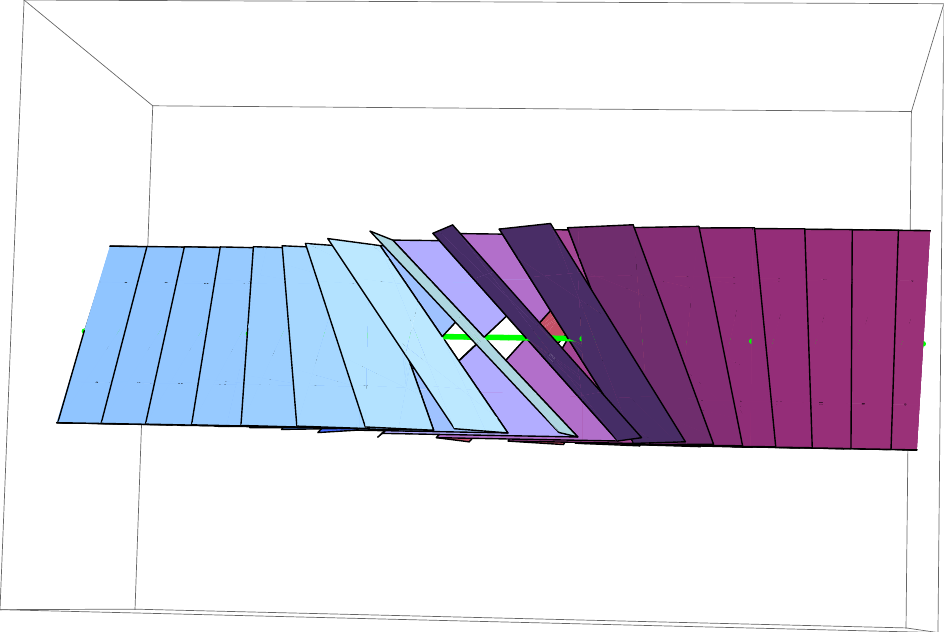}
  \end{subfigure}
  \begin{subfigure}[b]{0.6\textwidth}
      \includegraphics[scale=1]{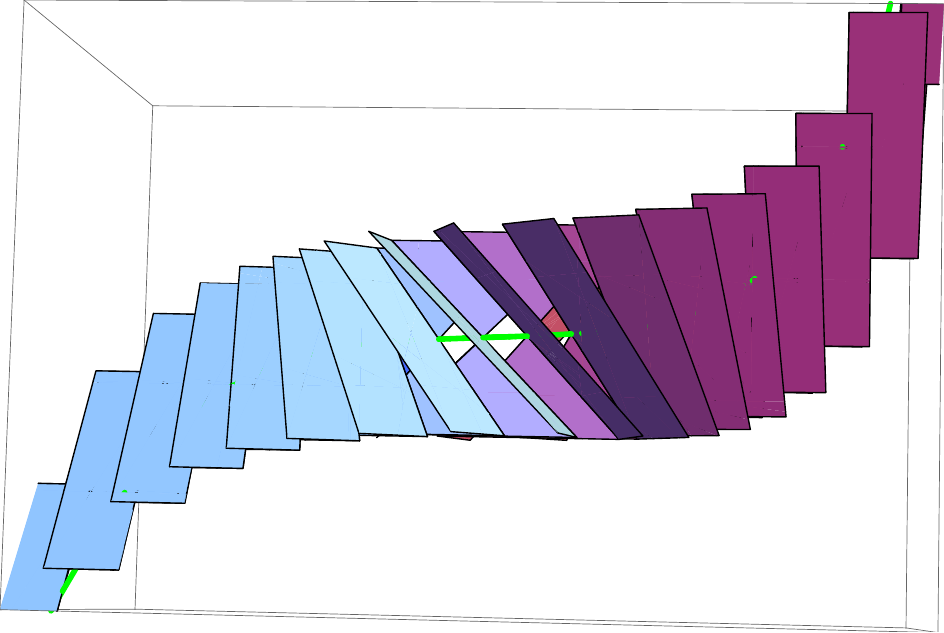}
  \end{subfigure}
    \caption{The first picture illustrates another view of
    Figure~\ref{fig:foliation}, where the vertex path
    is the invariant axis of an affine deformation $\gamma_t$
    of $\xi_t$. The second picture illustrates a crooked folation
    where the vertex path $p_t$ is a generic orbit of
    $\gamma_t$.
        }
\label{fig:vertexpaths}
\end{figure}

\end{document}